\documentclass[a4paper,twoside,10pt]{article}

\usepackage{amsmath}
\usepackage{amsthm}
\usepackage{amssymb}
\usepackage[english]{babel}
\usepackage[utf8]{inputenc}
\usepackage{lmodern,color}
\usepackage[T1]{fontenc}
\usepackage{indentfirst}
\usepackage[percent]{overpic}
\usepackage[a4paper,top=3cm,bottom=3cm,left=3cm,right=3cm,%
bindingoffset=5mm]{geometry}
\usepackage{booktabs}
\usepackage[compatibility=false]{subcaption}
\usepackage{emp}
\usepackage{braket}
\usepackage{chngpage}
\usepackage{cases}
\usepackage{graphicx}
\usepackage{epstopdf}
\usepackage{tabularx}
\usepackage{multirow}
\usepackage{enumerate}
\usepackage{todonotes}
\usepackage{authblk}
\usepackage{mdframed}

\newcommand{\numberset}{\mathbb}

\newcommand{\N}{\numberset{N}}
\newcommand{\R}{\numberset{R}}

\newcommand{\Pk}{\numberset{P}}
\newcommand{\Pkc}{\numberset{P}^{\rm cont}}

\newcommand{\diver}{{\rm div}}
\newcommand{\bdiver}{\boldsymbol{\diver}}
\newcommand{\bnabla}{\boldsymbol{\nabla}}

\newcommand{\bcurl}{\boldsymbol{{\rm curl}}}
\newcommand{\beps}{\boldsymbol{\epsilon}}

\newcommand{\jump}[1]{\lbrack\!\lbrack\,#1\,\rbrack\!\rbrack}
\newcommand{\media}[1]{\left\{\!\left\{\,#1\,\right\}\!\right\}}

\newcommand{\normas}[1]{\left\|#1\right\|_{\rm S}}
\newcommand{\normam}[1]{\left\|#1\right\|_{\rm M}}
\newcommand{\normaupw}[1]{\left|#1\right|_{\rm upw}}
\newcommand{\normacip}[1]{\left|#1\right|_{\rm cip}}
\newcommand{\normacurl}[1]{\left|#1\right|_{\rm \bcurl}}
\newcommand{\normastab}[1]{\left\|#1\right\|_{\rm stab}}

\newcommand{\normai}[1]{{\left\vert\kern-0.25ex\left\vert\kern-0.25ex\left\vert#1\right\vert\kern-0.25ex\right\vert\kern-0.25ex\right\vert}}

\newcommand{\nn}{\boldsymbol{n}}

\newcommand{\Edges}{\Sigma_h}
\newcommand{\EdgesB}{\Sigma_h^{\partial}}
\newcommand{\EdgesI}{\Sigma_h^{\rm {int}}}
\newcommand{\EdgesE}{\Sigma_h^E}

\renewcommand{\epsilon}{\varepsilon}
\renewcommand{\theta}{\vartheta}
\renewcommand{\rho}{\varrho}
\renewcommand{\phi}{\varphi}
\newcommand{\zz}{\boldsymbol{\zeta}}

\newcommand{\BB}{\boldsymbol{B}}
\newcommand{\HH}{\boldsymbol{H}}
\newcommand{\GG}{\boldsymbol{G}}
\newcommand{\WW}{\boldsymbol{W}}
\newcommand{\TT}{\boldsymbol{\Theta}}

\newcommand{\uu}{\boldsymbol{u}}
\newcommand{\vv}{\boldsymbol{v}}
\newcommand{\ww}{\boldsymbol{w}}
\newcommand{\cc}{\boldsymbol{\chi}}
\newcommand{\ff}{\boldsymbol{f}}

\newcommand{\pp}{\boldsymbol{p}}
\newcommand{\xx}{\boldsymbol{x}}

\newcommand{\pss}{\sigma_{\rm S}}
\newcommand{\psm}{\sigma_{\rm M}}
\newcommand{\ns}{\nu_{\rm S}}
\newcommand{\nm}{\nu_{\rm M}}
\newcommand{\bdma}{\mu_a}
\newcommand{\bdmc}{\mu_c}
\newcommand{\bdmJ}{\mu_{J_1}}
\newcommand{\bdmJJ}{\mu_{J_2}}

\newcommand{\WWc}{\boldsymbol{W}}
\newcommand{\WWd}{\boldsymbol{W}^h_k}
\newcommand{\Hunozero}{\boldsymbol{H}^1_0(\Omega)}
\newcommand{\VVc}{\boldsymbol{V}}

\newcommand{\VVd}{\boldsymbol{V}^h_k}
\newcommand{\ZZc}{\boldsymbol{Z}}

\newcommand{\ZZd}{\boldsymbol{Z}^h_k}
\newcommand{\Qc}{Q}
\newcommand{\Qd}{Q^h_{k}}

\newcommand{\am}{a^{\rm M}}
\newcommand{\as}{a^{\rm S}}

\newcommand{\ash}{a^{\rm S}_h}
\newcommand{\astab}{\mathcal{A}_{\rm stab}}

\newcommand{\ccoe}{c_{\rm coe}}

\newcommand{\intcip}{\mathcal{I}_{\mathcal{O}}}
\newcommand{\Pzerok}[1]{\Pi_{#1}}
\newcommand{\PWWd}{\mathcal{I}_{\WWc}}
\newcommand{\PVVd}{\mathcal{I}_{\VVc}}

\newcommand{\omegaEh}{\Omega_h^{M}}
\newcommand{\omegaz}{\omega_{\boldsymbol{\zeta}}}

\newcommand{\uui}{\PVVd \uu}
\newcommand{\vvi}{\PVVd \vv}
\newcommand{\eei}{\boldsymbol{e}_{\mathcal{I}}}
\newcommand{\eeh}{\boldsymbol{e}_h}
\newcommand{\BBi}{\PWWd \BB}

\newcommand{\EEi}{\boldsymbol{E}_{\mathcal{I}}}
\newcommand{\EEh}{\boldsymbol{E}_h}
\newcommand{\regu}{s}
\newcommand{\regb}{r}

\newcommand{\ls}{\Lambda_{\rm S}}
\newcommand{\gm}{\Gamma_{\rm M}}
\newcommand{\gs}{\Gamma_{\rm S}}
\newcommand{\ts}{\Phi_{\rm S}}

\newtheorem{theorem}{Theorem}
\newtheorem{proposition}[theorem]{Proposition}
\newtheorem{lemma}[theorem]{Lemma}
\newtheorem{remark}[theorem]{Remark}

\author[1,2]{{{L. Beir\~ao da Veiga}} 
\thanks{lourenco.beirao@unimib.it}}
\author[3,2]{{{C. Lovadina}} 
\thanks{carlo.lovadina@unimi.it}}
\author[1]{{{M. Trezzi}} 
\thanks{manuel.trezzi@unimib.it}}

\affil[1]{Dipartimento di Matematica e Applicazioni, Universit\`a degli Studi di Milano Bicocca, Via Roberto Cozzi 55 - 20125 Milano, Italy}  
\affil[2]{IMATI-CNR, Via Adolfo Ferrata 5 - 27100 Pavia, Italy}
\affil[3]{Dipartimento di Matematica ``F. Enriques'', Universit\`a degli Studi di Milano, Via Cesare Saldini 50 - 20133 Milano, Italy}

\title{A robust finite element method for linearized magnetohydrodynamics on general domains}

\begin{document}

\maketitle

\begin{abstract}
\noindent 
We generalize and improve the finite element method for linearized Magnetohydrodynamics introduced in \cite{BDV:2024}. The main novelty is that the proposed scheme is able to handle also non-convex domains and less regular solutions. The method is proved to be pressure robust and quasi-robust with respect to both fluid and magnetic Reynolds numbers. 
A set of numerical tests confirms our theoretical findings.
\end{abstract}

\section{Introduction}
Recently, the field of magnetohydrodynamics (MHD) has garnered increasing attention within the computational mathematics community. MHD equations are relevant in the study of plasmas and liquid metals, finding applications across geophysics, astrophysics, and engineering. The combination of fluid dynamics and electromagnetism equations results in a variety of models. These models have different formulations and lead to several finite element choices, each with its own advantages and disadvantages. Examples of relevant works in this area include  \cite{badia,VEM,da2023robust,dong,Gerbeau2,Greif,Guermond,Hiptmair,Houston,Perugia2,Perugia1,Prohl,Schotzau,Si:2015:TAD,Wacker}
(this list is not exhaustive).

In this paper we consider, as in \cite{BDV:2024}, the linearized version of the following three field formulation of the MHD problem:
\begin{equation}
\label{eq:non-linear primale}
\left\{
\begin{aligned}
 & \text{ find $(\uu, p, \BB)$ such that}\\
 &\begin{aligned}
 \uu_t -  \ns \, \bdiver (\beps (\uu) ) + (\bnabla \uu ) \,\uu    +  \BB \times \bcurl(\BB)  - \nabla p &= \ff \qquad  & &\text{in $\Omega \times I$,} 
 \\
 \diver \, \uu &= 0 \qquad & &\text{in $\Omega \times I$,} 
 \\
 \BB_t +  \nm \, \bcurl (\bcurl (\BB) ) - \bcurl(\uu \times \BB)  &= \GG \qquad  & &\text{in $\Omega \times I$,} 
 \\
 \diver \, \BB &= 0 \qquad & &\text{in $\Omega \times I$.} 
 \end{aligned}
 \end{aligned}
 \right.
 \end{equation}
Above, $\Omega$ is a spatial three-dimensional domain, while $I$ is a time interval. Moreover, $\uu$, $p$ and $\BB$ are the velocity, pressure and magnetic fields, respectively; the parameters $\nu_S$ an $\nu_M$ represent the fluid and the magnetic diffusion coefficients, respectively. 
We recall that in several engineering and physical applications, the scaled parameters  $\nu_S$ and, potentially, $\nu_M$ are quite small. For instance, in aluminum electrolysis, $\nu_S$ is approximately {\tt 1e-5} and $\nu_M$ is around {\tt 1e-1} (see, for example, \cite{Armero,Berton}). Even smaller values of $\nu_M$ are encountered, for instance, in geophysics and space weather prediction problems. When these parameters are small, it is well-known that standard finite elements become unstable and stabilization techniques are necessary to achieve reliable numerical results. 
In the literature there are only a few finite element schemes that are (provably) robust for small values of $\nu_S,\nu_M$, such as \cite{Gerbeau2,Houston,Wacker,BDV:2024} for the linear case and, to our knowledge, only \cite{B.D.V.nonlinear} for the fully nonlinear case. 

The presented scheme takes the initial steps from the method in \cite{BDV:2024} but, contrary to such contribution which is focused to convex domains $\Omega$, we here allow for general domains, possibly non-convex. We highlight that this extension is not merely a straightforward adaptation of the tools used for the convex case. Instead, considering non-convex domains implies deeper changes, including a different framework for the variational formulation of Problem \eqref{eq:non-linear primale}, cf. Remark 2.2 in \cite{BDV:2024}. 

However, the finite element schemes we study in this paper maintain most of the significant features of the one detailed in \cite{BDV:2024}, namely:
 \begin{enumerate}
     \item pressure robustness, i.e. the fluid velocity error does not depend on the pressure error; 
     \item quasi-robustness with respect to dominant advection, i.e. assuming a sufficiently regular solution, the estimates are independent of both $\nu_S$ and $\nu_M$ in an error norm that controls convection);
     \item suitability for extension to the time-dependent case using a standard time-stepping scheme, unlike SUPG or similar methods.
 \end{enumerate}
To achieve these objectives also for general domains, rather than the space $\HH^1$ we first consider a variation formulation involving the space $\HH(\bcurl)$ for the magnetic field, and select the N\'ed\'elec finite element of second type for its discretization. 
As a consequence, the divergence free condition for the magnetic field $\BB$, see fourth equation of \eqref{eq:non-linear primale}, cannot be directly inserted in the variational formulation, but it is recovered by the assumption that $\GG$ is orthogonal to the gradients.

The fluid velocity and pressure are essentially treated as in \cite{BDV:2024}, as well as the same strategies to ensure robustness with respect to all regimes are employed. More precisely:
\begin{itemize}
     \item for pressure robustness, velocity is approximated by means of $\HH(\diver)$-conforming $\boldsymbol{{\rm BDM}}$ elements, while standard discontinuous piecewise polynomials are used for the pressure discretization. Hence, the velocity approximation space gives rise to a non-conforming scheme, which needs to be stabilized. Stability is then recovered by adding face terms in the spirit of the DG methods;
     \item in order to be robust with respect to the fluid diffusion parameter, a typical upwinding DG term is included in the variational formulation;
     \item to deal with the fluid-magnetic coupling, and to be robust also with respect to the magnetic diffusion parameter, suitable Continuous Interior Penalty (CIP) terms are introduced. We remark that, though following a similar strategy, the CIP terms we consider here are slightly different from the corresponding ones detailed in \cite{BDV:2024}.   
\end{itemize}

Even though we borrow some techniques and results from \cite{BDV:2024}, we highlight that our analysis represents a significant improvement over the one developed in that paper. 
In particular, it must be observed that the technique adopted in \cite{BDV:2024} to exploit the CIP (Continuous Interior Penalty) stabilization in the convergence proofs is not suitable for the present setting. The main reason is that such an approach, based on a peculiar orthogonality property, prevents the use of specific N\'ed\'elec interpolants, which are instead needed to handle the kind of solutions that are encountered in magnetic problems on non-convex domains. Therefore we take a different route, and prove a stronger stability with respect to a norm associated with the linearization of the convective term $\bcurl(\uu \times \BB)$. This approach, although initially more involved, has two advantages: 
\begin{enumerate}
    \item it allows more freedom in the choice of the interpolant during the convergence analysis, as commented here above;
    \item it shows stability in a stronger norm which also includes a magneto-advective term, thus providing a solid theoretical justification of what it is observed in the numerical experiments, i.e. the scheme robustness in the magneto-advective regime.
\end{enumerate}

In addition, we present some numerical results that clearly show the failure of the method described in \cite{BDV:2024} (well-suited for convex domains), when applied to a kind of 3D L-shaped domain. 




The paper is organized as follows. Section \ref{sec:model} introduces the continuous problem under consideration, i.e. the linearized version of Problem \eqref{eq:non-linear primale}, together with a suitable variational formulation, capable to deal with possibly non-convex domains.  
Section \ref{sec:prelim} presents notation and preliminary results, useful for the subsequent study. 
Section \ref{sec:discrete} details our proposed stabilized scheme, by introducing the necessary discrete spaces and forms.
Section \ref{sec:theoretical} develops the theoretical convergence analysis: first, we prove a stability result involving norms that can also handle the advection-dominated regime; secondly, the optimal error estimates in those norms are derived. 
The numerical tests are presented in Section \ref{sec:numerical}. As already mentioned, in addition to support the theoretical estimates, the numerical experiments give evidence that $\HH^1$-conforming approximations for the magnetic field are unable to converge when general non-convex domains are involved. 

Finally, throughout the paper we use standard notations for functional (Sobolev) spaces and their norms/seminorms. However, we explicitly recall the definition of the following spaces.

\begin{equation}\label{eq:spaces}
\begin{aligned}
    &L^2_0(\Omega):= \{v \in L^2(\Omega) \,\,\,\text{s.t.} \,\,\, \int_\Omega v = 0 \}\ ,\\
    &\HH_0(\diver, \Omega) := \{\vv \in \boldsymbol{L}^2(\Omega) \,\,\,\text{s.t.} \,\,\, \diver \,\vv \in L^2(\Omega) \,\,\, \text{and} \,\,\, \vv \cdot \nn = 0 \,\,\, \text{on $\partial \Omega$} \}\ ,\\
    &\HH(\bcurl, \Omega) := \{\HH \in \boldsymbol{L}^2(\Omega) \,\,\,\text{s.t.} \,\,\, \bcurl (\HH) \in \boldsymbol{L}^2(\Omega)  \}\ ,\\
    &\HH^r(\bcurl, \Omega) := \{\HH \in \boldsymbol{H}^r(\Omega) \,\,\,\text{s.t.} \,\,\, \bcurl (\HH) \in \boldsymbol{H}^r(\Omega)  \} \ \ \ (r> 0)\, .
\end{aligned}
\end{equation}

\section{Model problem}\label{sec:model}
A linearized version of problem \eqref{eq:non-linear primale} reads as
\begin{equation}
\label{eq:linear primale}
\left\{
\begin{aligned}
& \text{ find $(\uu, p, \BB)$ such that}\\
&\begin{aligned}
\pss \, \uu -  \ns \, \bdiver (\beps (\uu) ) + (\bnabla \uu ) \,\cc    +  \TT \times \bcurl(\BB) - \nabla p &= \ff \qquad  & &\text{in $\Omega$,} 
\\
\diver \, \uu &= 0 \qquad & &\text{in $\Omega$,} 
\\
\psm \, \BB +  \nm \, \bcurl (\bcurl (\BB) ) - \bcurl(\uu \times \TT)  &= \GG \qquad  & &\text{in $\Omega$,} 
\\
\diver \, \BB &= 0 \qquad & &\text{in $\Omega$,} 
\end{aligned}
\end{aligned}
\right.
\end{equation}
coupled with the homogeneous boundary conditions
\begin{equation}
\label{eq:linear bc cond}
\uu = 0 \,, \quad 
\BB \cdot \nn = 0 \,,  \quad 
\bcurl(\BB) \times \nn = 0  \quad \text{on }\partial \Omega \, .
\end{equation}
Here above we denote by $\pss$ and $\psm$ the reaction coefficients (constant for simplicity), and with $\TT, \cc$ the advective fields of the fluid and the magnetic field, respectively. For the time being, we assume $\TT \in \HH_0(\diver,\Omega) \cap  \boldsymbol{L}^3(\Omega)$, and $\cc \in \HH(\bcurl,\Omega) \cap  \boldsymbol{L}^3(\Omega)$ , with $\diver\,\cc = 0$.

We are interested in deriving a variational formulation of \eqref{eq:linear primale}. 
We introduce the spaces
\begin{equation}
\label{eq:spazi_c}
\VVc := \Hunozero \,, 
\quad
\WWc := \HH(\bcurl, \Omega)\,,
\quad
\Qc  := L^2_0(\Omega)  \,,
\end{equation}
that represent the velocity fields space, the magnetic induction space, and the space of pressures, respectively. 
Contrary to the convex setting, we emphasize that an  $\HH^1$-conforming space for the magnetic induction cannot be assumed. In the convex case, controlling both the divergence and the $\bcurl$ of a function ensures control over its gradient, which implies that the function belongs to $\HH^1(\Omega)$.
However, when the domain is non-convex, this control is weaker, and we can only conclude that the function lies in $\HH^{s}(\Omega)$ for some suitable $s > \frac{1}{2}$ depending on $\Omega$. 
In the velocity space, we introduce the following bilinear forms
\begin{equation}
\label{eq:forme_c1}
\as(\uu,  \vv) :=  (\beps(\uu), \, \beps (\vv) ) \,,
\qquad
c(\uu, \vv) :=  \left( ( \bnabla \uu ) \, \cc ,\, \vv  \right) \,,
\end{equation}
where the first one is obtained by multiplying the term $\diver(\beps(\uu))$ by a test function $\vv$, integrating by parts and using the boundary conditions. On $\WWc$, we define the form
\[
\am(\BB,  \HH) = (\bcurl (\BB), \bcurl (\HH)) \, .
\]
The coupling is represented by two forms: the first one controls the interaction between velocity and pressure, while the second one governs the coupling between velocity and magnetic induction
\begin{equation}
\label{eq:forme_c2}
\begin{aligned}
b(\vv, q) :=  (\diver \vv, \, q) \,, \qquad
d(\HH, \vv) :=  ( \bcurl (\HH ) \times \TT ,\, \vv) \,.
\end{aligned}
\end{equation}
We also introduce the kernel of the  bilinear form $b(\cdot,\cdot)$
that corresponds to the functions in $\VVc$ with vanishing divergence
\begin{equation}
\label{eq:Z}
\ZZc := \{ \vv \in \VVc \quad \text{s.t.} \quad \diver \, \vv = 0  \}\,.
\end{equation}
Then, we consider the following variational problem
\begin{equation}
\label{eq:linear variazionale}
\left\{
\begin{aligned}
& \text{find $(\uu, p, \BB) \in \VVc \times \Qc \times \WWc$,  such that} 
\\
&\begin{aligned}
\pss \, (\uu, \vv) + \ns \, \as(\uu, \vv) + c(\uu, \vv)  
-d(\BB, \vv)  + b(\vv, p) 
&= 
(\ff, \vv)  
&\,\,\, & \text{for all $\vv \in \VVc$,} 
\\
b(\uu, q) &= 0 
&\,\,\, & \text{for all $q \in \Qc$,}
\\
\psm \, (\BB, \HH) + \nm \, \am(\BB, \HH) + d(\HH, \uu) 
&= 
(\GG, \HH)  
&\,\,\, & \text{for all $\HH \in \WWc$.}
\end{aligned}
\end{aligned}
\right.
\end{equation}
As usual, we suppose that $\GG$ is $\boldsymbol{L}^2$-orthogonal to the gradients. We also remark that the boundary conditions on $\BB$, see \eqref{eq:linear bc cond}, are here weakly imposed. 
This problem is well posed by the Lax-Milgram lemma and standard theory of mixed problems.

\section{Notations and preliminary results}
\label{sec:prelim}

In this section, we introduce the notation and some useful results that will be used throughout the rest of the paper.
We consider a family of decompositions $\{ \Omega_h\}_h$ of the domain $\Omega$ into non-overlapping tetrahedrons $E$. 
We denote with $h_E$ the diameter of each element and with $h := \max_{E \in \Omega_h} h_E$ the  mesh size of $\Omega_h$.
Let $\mathcal{N}_h$ be the set of vertices in $\Omega_h$;
given $\zeta \in \mathcal N_h$, we denote with $\omegaz$ the collection of elements $E$ in $\Omega_h$ such that $\zeta \in \partial E$.
We adopt the same mesh assumptions as in \cite{BDV:2024}. The first assumption is classical in the FEM framework. It states that the elements in the decomposition are not excessively stretched, which is necessary to achieve optimal approximation results.

\smallskip
\noindent
\textbf{(MA1) Shape regularity assumption:}

\noindent
The mesh family $\left\{ \Omega_h \right\}_h$ is shape regular: it exists a  positive  constant $c_{\rm{M}}$ such that each element $E \in \{ \Omega_h \}_h$ is star shaped with respect to a ball of radius $\rho_E$ with  $h_E \leq c_{\rm{M}} \rho_E$. 

\smallskip
\noindent
The second assumption is specific to the case 
$k=1$. While this assumption is required for theoretical purposes, it is not overly restrictive.\\
\textbf{(MA2) Mesh agglomeration with stars macroelements:}

There exists a family of conforming meshes $\{ \widetilde{\Omega}_h \}_h$ of $\Omega$ with the following properties:
(i) it exists a positive constant $\widetilde{c}_{\rm{M}}$ such that each element $M \in \widetilde{\Omega}_h$ is a finite  (connected) agglomeration of elements in $\Omega_h$, i.e., it exists $\omegaEh \subset \Omega_h$ with ${\rm{card}}(\omegaEh) \leq \widetilde{c}_{\rm{M}}$ and $M = \cup_{E \in \omegaEh} E$;
(ii) for any $M \in \widetilde{\Omega}_h$ it exists $\zz \in \mathcal{N}_h$ such that $\omegaz \subseteq M$.
\smallskip
\noindent
We observe that assumption \textbf{(MA1)} readily implies that the mesh is locally quasi-uniform. This means that the diameter of a face is uniformly bounded by the diameters of the elements containing the face and that the sizes of two adjacent elements are comparable.

The set of all the faces in the mesh $\Omega_h$ is denoted by $\Edges$. This set is divided into internal faces $\EdgesI$ and boundary faces $\EdgesB$. Given an element $E\in\Omega_h$, the set of faces of that element is denoted with $\EdgesE$. 
Furthermore, given an element $E \in \Omega_h$, we denote with $\nn^E$ the unit outward normal to that element, while $\nn^f$ denotes the unit normal to a face $f$. In particular, if $f \in \EdgesI$ we have that $\nn^f =  \nn^E$ or $\nn^f = - \nn^E$, while if $ f \in \EdgesB$ it holds $\nn^f = \nn^E = \nn.$
The jump and the average operators on $f \in \EdgesI$ are defined for every piecewise continuous function w.r.t. $\Omega_h$ respectively by
\[
\begin{aligned}
\jump{\phi}_f(\xx) &:=
\lim_{s \to 0^+} \left( \phi(\xx - s \nn_f) - \phi(\xx + s \nn_f) \right)
\\
\media{\phi}_f(\xx) &:= 
\frac{1}{2}\lim_{s \to 0^+} \left( \phi(\xx - s \nn_f) + \phi(\xx + s \nn_f) \right)
\end{aligned}
\]
and $\jump{\phi}_f(\xx) = \media{\phi}_f(\xx) = \phi(\xx)$ on $f \in \EdgesB$.
Given a positive integer $m \in \N$ and a subset $S \subseteq \Omega_h$, we introduce the standard polynomial spaces:
\begin{itemize}
\item $\Pk_m(\omega)$ is the set of polynomials on $\omega$ of degree $\leq m$, with $\omega$ a generic set;
\item $\Pk_m(S) := \{q \in L^2\bigl(\cup_{E \in S}E \bigr) \quad \text{s.t.} \quad q|_{E} \in  \Pk_m(E) \quad \text{for all $E \in S$}\}$ is the set of piecewise discontinuous polynomials;
\item $\Pkc_m(S) := \Pk_m(S) \cap C^0\bigl(\cup_{E \in S}E \bigr)$ is the set of piecewise continuous polynomial.
\end{itemize}
Given $s \in \R^+$ and  $p \in [1,+\infty]$, we define the standard broken Sobolev spaces as:
\begin{itemize}
\item $W^s_p(\Omega_h) := \{\phi \in L^2(\Omega) \quad \text{s.t.} \quad \phi|_{E} \in  W^s_p(E) \quad \text{for all $E \in \Omega_h$}\}$,
\end{itemize}
equipped with  the usual broken norm 
$\Vert \cdot \Vert_{W^s_p(\Omega_h)}$ 
and seminorm 
$\vert \cdot \vert_{W^s_p(\Omega_h)}$. Similarly, we denote with $\HH^{\regb}(\bcurl,\Omega_h)$ the space of vectorial $\boldsymbol L^2(\Omega)$ functions $\HH$ such that $\HH_{|E}\in \HH^{\regb}(\bcurl,E)$ for every $E\in\Omega_h$, endowed with the corresponding broken norm.
Furthermore, we introduce the following space
\begin{itemize}
\item $\mathbb{O}_{k-1}(\Omega_h) := 
\Pkc_{k-1}(\Omega_h)$  for $k>1$, or 
$\mathbb{O}_{k-1}(\Omega_h) := \Pk_{0}(\widetilde{\Omega}_h)$  for $k=1$.
\end{itemize}
Fix a face $f\in\EdgesI$ shared by two elements $E^+$ and $E^-$.
For any $p \in [1,\infty]$, we will use the notation
\[
\| V\|_{L^{p}(f^\pm)} := \max 
\{ 
\| V\|_{L^{p}(f^+)}, 
\| V\|_{L^{p}(f^-)}
\} 
\]
to denote the $L^p$ norm of a function that is not continuous across $f$. Above $f^+$ (resp., $f^-$) is the face $f$ considered as part of the boundary $\partial E^+$ (resp., $\partial E^-$).
In the paper, the symbols $\gtrsim$ and $\lesssim$ denote inequalities that hold up to a constant depending solely on the order of the method $k$, the domain $\Omega$, and the regularity of the mesh $\Omega_h$.
This constant does not depend on the model parameters $\psm$, $\pss$, $\ns$, $\nm$, $\cc$, and $\TT$, nor on the loadings $\mathbf{f}$, $\GG$, or the solution $(\uu, p, \BB)$. \\
We conclude this section by mentioning some useful preliminary results.
\begin{lemma}[Trace inequality]
\label{lm:trace}
Under the mesh assumption \textbf{(MA1)}, for any $E \in \Omega_h$ and for any  function $v \in H^s(E)$ with $  \frac{1}{2} < s \leq 1$, it holds 
\[
\sum_{f \in \EdgesE}\|v\|^2_{f} \lesssim h_E^{-1}\|v\|^2_{E} + h_E^{2s-1}| v|^2_{s,E} \,.
\]
\end{lemma}

\begin{lemma}[Bramble-Hilbert]
\label{lm:bramble}
Under the mesh assumption \textbf{(MA1)}, let $m \in {\mathbb N}$.  We denote with
$\Pzerok{m} \colon L^2(\Omega_h) \to \Pk_m(\Omega_h)$ the  $L^2$-projection operator onto the space of piecewise polynomial functions. For any $E \in \Omega_h$ and for any  smooth enough function $\phi$ defined on $\Omega$, it holds 
\[
\Vert\phi - \Pzerok{m} \phi \Vert_{W^r_p(E)} \lesssim h_E^{s-r} \vert \phi \vert_{W^s_p(E)} 
\qquad  \text{$s,r \in \N$, $r \leq s \leq m+1$, $p \in [1, \infty]$.}
\]
\end{lemma}

\begin{lemma}[Inverse estimate]
\label{lm:inverse}
Under the mesh assumption \textbf{(MA1)}, let $m \in {\mathbb N}$. 
Then for any $E \in \Omega_h$ and for any $p_m \in \Pk_m(E)$ it holds 
\[
\|p_m\|_{W^s_p(E)} \lesssim h_E^{-s} \|p_m\|_{L^p(E)} 
\]
where the involved constant only depends on $m,s,p,c_{\rm{M}}$.
\end{lemma}
We consider an interpolation operator that maps sufficiently smooth piecewise discontinuous functions into the space of more regular functions. This interpolation operator is commonly referred to in the literature as an averaging operator \cite{burman:2004, EG:2017, BDV:2024}. 
Specifically, our goal is to take a piecewise polynomial $p$ of degree $k-1$ and return a piecewise polynomial of the same degree that belongs to $\mathbb{O}_{k-1}(\Omega_h)$. The proof of the following proposition can be found in \cite{BDV:2024}.
\begin{proposition} \label{prp:oswald}
Let Assumption \textbf{(MA1)} hold. Furthermore, if $k=1$ let also Assumption \textbf{(MA2)} hold. Then, it exists a projection operator $\intcip \colon \Pk_{k-1}(\Omega_h) \to \mathbb{O}_{k-1}(\Omega_h)$
such that for any $p_{k-1} \in \Pk_{k-1}(\Omega_h)$   the following holds:
\[
\sum_{E \in \Omega_h}h_E^2  \Vert (I- \intcip) p_{k-1} \Vert_E^2 \lesssim 
\sum_{f \in \EdgesI}
h_f^3 \Vert \jump{p_{k-1}}_f\Vert_f^2
\, .
\]
Furthermore, using a triangular inequality combined with a trace inequality, and the regularity of the mesh, it holds
\[
\sum_{E \in \Omega_h}h_E^2  \Vert (I- \intcip) p_{k-1} \Vert_E^2 \lesssim 
\sum_{E \in \Omega_h}
h_E^2 \Vert p_{k-1}\Vert_E^2
\, .
\]
\end{proposition}


\section{The stabilized finite element method}
\label{sec:discrete}
\subsection{Discrete spaces}
Given a positive integer $k$ that represents the order of the method, we introduce the following polynomial spaces
\begin{equation}
\label{eq:spazi_d}
\begin{aligned}
\VVd := [\Pk_k(\Omega_h)]^3 \cap &\HH_0(\diver,\Omega) \, , \quad  
\WWd := [\Pk_k(\Omega_h)]^3 \cap \WWc \, , \\
&\Qd  := \Pk_{k-1}(\Omega_h) \cap Q \, .
\end{aligned}
\end{equation}
For the discrete velocity field $\VVd$, we select standard $\HH(\diver)$-conforming Brezzi-Douglas-Marini $\boldsymbol{{\rm BDM}}_k$ elements, while for the magnetic induction we select Nédéléc elements of the second kind.
We introduce the discrete kernel space
\[
\ZZd := \{ \vv_h \in \VVd \quad \text{s.t.} \quad \diver \, \vv_h = 0  \}\, ,
\]
and thanks to the choice of the $\boldsymbol{{\rm BDM}}_k$ elements that maintain the pressure robustness of the method, we have the inclusion $\ZZd \subseteq \ZZc$.
We have to introduce two interpolation operators. The first one maps functions in $\HH^1(\Omega)$ into the space $\VVd$, mapping the continuous kernel into the discrete kernel, see for instance \cite{boffi-brezzi-fortin:book,ernguer-book-1}. 
\begin{lemma}[Interpolation operator on $\VVd$]
\label{lm:int-v}
Under the Assumption \textbf{(MA1)} let $\PVVd \colon \VVc \to \VVd$ be the standard degree-of-freedom interpolation operator defined in the BDM space.
Then: 

\noindent
$(i)$ if $\vv \in \ZZc$ then $\vvi \in \ZZd$;

\noindent
$(ii)$ for any $\vv \in \ZZc$
\begin{equation}
\label{eq:orth-v}
\left(\vv - \, \vvi, \,  \pp_{k-1} \right) = 0 \quad \text{for all $\pp_{k-1} \in [\Pk_{k-1}(\Omega_h)]^3$;}
\end{equation}

\noindent
$(iii)$ for any $\vv \in \VVc \cap \HH^{s+1}(\Omega_h)$, with $0 \leq s \leq k$, for all $E \in \Omega_h$, it holds
\begin{equation}
\label{eq:int-v}
\vert \vv - \vvi \vert_{m,E} \lesssim h_E^{s+1-m} \vert \vv \vert_{s+1,E} 
\qquad \text{for $0\leq m\leq s+1$.}
\end{equation}
\end{lemma}
The second interpolation concerns the spaces $\WWc$, see for instance \cite{AV:1999}. 
A notable property is that this interpolant also approximates the $\bcurl$ of the function.
\begin{lemma}[Interpolation operator on $\WWd$]\label{lm:int-w}
Under the Assumption \textbf{(MA1)} let $\PWWd \colon \WWc \to \WWd$ be the interpolation operator of \cite{AV:1999}. 
If the function satisfies also $\HH \in \HH^r(\bcurl, \Omega)$ for $\frac{1}{2} < r \leq k+1$, the following estimates hold
\[
\begin{aligned}
&| \HH - \PWWd \HH |_{m,E} \lesssim h_E^{r-m} | \HH |_{\HH^r(\bcurl, E)} \quad &\text{for $0\leq m\leq r$}\, ,  \\
&| \bcurl(\HH - \PWWd \HH) |_{m,E} \lesssim h_E^{\tilde{r}-m}| \bcurl(\HH) |_{\tilde{r}, E} \quad &\text{for $0\leq m\leq \tilde r$}\, ,
\end{aligned}
\]
where $\tilde r := \min \{ r,k \}$.
\end{lemma}

\subsection{Discrete problem}

From now on, we assume that the advective fields satisfy the following regularity requirements, in addition to the minimal conditions detailed at the beginning of Section \ref{sec:model}:
\begin{itemize}
\item the advective velocity field $\cc \in \boldsymbol{W}^1_ \infty(\Omega_{h0})  \cap \HH_0(\diver, \Omega)$, 
\item the advective magnetic induction $\TT \in \boldsymbol{W}^2_ \infty(\Omega_{h0})  \cap {\bf W}$,
\end{itemize}
where $\Omega_{h0}$ is a coarse mesh such that all the meshes in the family $\{ \Omega_h \}_h$ are a refinement of $\Omega_{h0}$. Note that our assumption is purposefully more general than assuming $\cc \in \boldsymbol{W}^1_ \infty(\Omega)$, $\TT \in \boldsymbol{W}^2_ \infty(\Omega)$ since it allows also for advective fields in the corresponding discrete spaces ($\VVd$ and $\WWd$, respectively). This is compatible with the situation when the linearized problem is thought of as originated by a fixed point iteration of the fully nonlinear magnetohydrodynamic equations. Our condition is set on $\Omega_{h0}$, and not on $\Omega_h$, simply because our data $\cc,\TT$ are given and should not depend on the mesh parameter $h$.

In addition, before introducing the discrete problem, we preliminary make the following assumption on the  solution $\uu$ of problem \eqref{eq:linear variazionale}.

\smallskip
\noindent
\textbf{(RA1) Regularity assumption for the consistency:}

\noindent
Let $\uu \in \ZZc$ be the velocity solution  of Problem \eqref{eq:linear variazionale}, then $\uu$ belongs to $\HH^{r}(\Omega)$ for some $r> 3/2$. \medskip

We consider the following discrete problem 
\begin{equation}
\label{eq:linear fem}
\left\{
\begin{aligned}
& \text{find $(\uu_h, p_h, \BB_h) \in \VVd \times \Qd \times \WWd$,  such that} 
\\
&\begin{aligned}
\pss \,(\uu_h, \vv_h) + \ns \, \ash(\uu_h, \vv_h) + c_h(\uu_h, \vv_h) + &
\\ 
-d(\BB_h, \vv_h)   + J_h(\uu_h, \vv_h) + b(\vv_h, p_h) 
&= 
(\ff, \vv_h)  
&\,\,\, & \text{for all $\vv_h \in \VVd$,} 
\\
b(\uu_h, q_h) &= 0 
&\,\,\, & \text{for all $q_h \in \Qd$,}
\\
\psm \,(\BB_h, \HH_h) + \nm \,\am(\BB_h, \HH_h) + d(\HH_h, \uu_h) 
&= 
(\GG, \HH_h)  
&\,\,\, & \text{for all $\HH_h \in \WWd$,}
\end{aligned}
\end{aligned}
\right.
\end{equation}
where the discrete bilinear forms are defined as:
\begin{equation}
\label{eq:forme_d}
\begin{aligned}
\ash(\uu_h,  \vv_h) &:=  (\beps_h(\uu_h) ,\, \beps_h(\vv_h))
- \sum_{f \in \Edges} (\media{\beps_h(\uu_h)\nn_f}_f ,\, \jump{\vv_h}_f)_f  +
\\
&- \sum_{f \in \Edges} (\jump{\uu_h}_f ,\, \media{\beps_h(\vv_h) \nn_f}_f)_f 
+ 
\bdma \sum_{f \in \Edges} h_f^{-1} (\jump{\uu_h}_f ,\,\jump{\vv_h}_f)_f \, ,
\\
c_h(\uu_h, \vv_h) &:=  (( \bnabla_h \uu_h ) \, \cc, \, \vv_h )
- \sum_{f \in \EdgesI} ( (\cc \cdot \nn_f) \jump{\uu_h}_f ,\, \media{\vv_h}_f)_f +
\\
& + 
\bdmc \sum_{f \in \EdgesI} (\vert \cc \cdot \nn_f \vert \jump{\uu_h}_f, \, \jump{\vv_h}_f )_f \, .
\end{aligned}
\end{equation}
To stabilize the method and obtain feasible numerical solutions in the hyperbolic limit, we include the following jump penalization term in the formulation of the discrete problem
\begin{equation}
\label{eq:J}
\begin{aligned}
J_h(\uu_h, \vv_h) &:=   \sum_{f \in \Edges} \bdmJ  \bigl(
\jump{\TT \times \uu_h}_f , \jump{\TT \times \vv_h}_f \bigr)_f \\
& \qquad +
\sum_{f \in \EdgesI} \bdmJJ h_f^2 
\bigl(\jump{\bcurl_h (\uu_h \times \TT)}_f  ,\jump{\bcurl_h (\vv_h \times \TT)}_f 
\bigr)_f \, ,
\end{aligned}
\end{equation}
where $\bdmJ$ and $\bdmJJ$ are two parameters that will be fixed later. In contrast to the standard CIP term commonly found in the literature \cite{burman:2004, BDV:2024}, we emphasize that our approach considers the jump of the $\bcurl$ rather than the jump of the full gradient.
\begin{remark}
It is possible to consider the following alternative $J_h(\cdot,\cdot)$:
\[
\begin{aligned}
J_h(\uu_h, \vv_h) &:=  \sum_{f \in \Edges} 
\Vert \TT\Vert_{\boldsymbol{L}^\infty(f^\pm)}^2 
 \bdmJ (\jump{\uu_h}_f ,  \jump{\vv_h}_f)_f \\
&\qquad + 
\sum_{f \in \EdgesI}
\Vert \TT\Vert_{\WW^1_\infty(f^\pm)}^2
\bdmJJ  h_f^2 (\jump{\bnabla_h \uu_h}_f  , \jump{\bnabla_h \vv_h}_f )_f \, .
\end{aligned}
\]
\end{remark}
Finally, we define the bilinear form
\begin{equation}
\label{eq:astab}
\begin{aligned}
\astab(\uu_h, \BB, \vv_h, \HH_h) := &
\pss (\uu_h, \vv_h) + \psm (\BB, \HH_h) + 
\ns \ash(\uu_h, \vv_h) + \nm \am(\BB, \HH_h) 
\\  
& 
+ c_h(\uu_h, \vv_h) -d(\BB, \vv_h)   + d(\HH_h, \uu_h) + J_h(\uu_h, \vv_h) \,.
\end{aligned}
\end{equation}
Problem \eqref{eq:linear fem} can be written also in the following pressure-independent form
\begin{equation}
\label{eq:Astabprob}
\left\{
\begin{aligned}
& \text{find $(\uu_h,\BB_h) \in \ZZd  \times \WWd$,  such that} 
\\
&\astab(\uu_h, \BB_h, \vv_h, \HH_h) = (\ff, \vv_h) + (\GG, \HH_h) \quad   \text{for all $\vv_h \in \ZZd$} , \, \text{for all $\HH_h \in \WWd$} \, .
\end{aligned}
\right.
\end{equation}

\begin{remark}
All the forms above are intended to be extendable to any sufficiently regular function. In particular, if the solution $\uu$ of the continuous problem \eqref{eq:linear variazionale} satisfies (RA1), the following consistency property holds
\begin{equation}\label{eq:consistency}
\astab(\uu-\uu_h, \BB-\BB_h, \vv_h, \HH_h) = 0 \quad   \text{for all $\vv_h \in \ZZd$} , \, \text{for all $\HH_h \in \WWd$} \, .
\end{equation}
\end{remark}

\begin{remark}\label{rem:Bh:kernel}
Exploiting the third equation in \eqref{eq:linear fem} and choosing $\HH_h$ as the gradient of a polynomial $p \in \Pkc_{k+1}(\Omega_h)$, we obtain that
\[
(\BB_h, \nabla p) = 0 \quad \forall p \in \Pkc_{k+1}(\Omega_h) \, ,
\]
which is the classical discrete divergence-free condition for Nédéléc elements.
\end{remark}

\section{Theoretical analysis}
\label{sec:theoretical}
\subsection{Inf-sup condition}

We introduce the following norms and seminorms, which depend on the equation parameters and on the mesh:
\begin{equation}
\label{eq:norme}
\begin{aligned}
\normas{\uu}^2 &:= 
\pss \Vert \uu \Vert^2 + \ns \Vert \beps_h(\uu) \Vert^2 + 
\ns \, \bdma\!\sum_{f \in \Edges} h_f^{-1} \Vert \jump{\uu}_f \Vert_{f}^2 \, ,
\\
\normaupw{\uu}^2 &:= \bdmc\!\sum_{f \in \EdgesI}  \Vert \vert \cc \cdot \nn_f \vert^{1/2} \jump{\uu}_f \Vert_{f}^2 \, ,
\\
\normacip{\uu}^2 &:=
\bdmJ\sum_{f \in \Edges}  \Vert \jump{\TT \times \uu}_f \Vert_{f}^2 \, ,
+
\bdmJJ\sum_{f \in \EdgesI}  h_f^2 \Vert \jump{\bcurl_h(\uu\times \TT_h)}_f \Vert_{f}^2\\
\normacurl{\uu}^2 &:=
\gamma^{-1} \, \sum_{E \in \Omega_h} h_E^2 \, \| \bcurl_h(\uu \times \TT_h) \|^2_{E} \, .
\end{aligned}
\end{equation}
where the global parameter $\gamma := \max\{h, \nm\}$ (hence $\gamma^{-1} := \min\{h^{-1},\nm^{-1}\}$) and $\TT_h$ is the best piecewise constant approximation of $\TT$ in $L^2(\Omega)$.

\begin{remark}\label{rem:scaling}
We are aware that $h$ and $\nm$ have different physical dimensions, so that in principle the comparison $\gamma := \max\{h, \nm\}$ would require a suitable preliminary scaling of these quantities. However, it is common practice in the literature concerning finite element analysis of advection-diffusion-reaction problems to ignore this aspect and assume that the involved quantities are already properly scaled. Here, we follow this approach.
\end{remark}

The stability norms are obtained by summing these norms and seminorms
\begin{equation}
\label{eq:normeVW}
\begin{aligned}
\normastab{\uu}^2 &:= 
\normas{\uu}^2 + \normaupw{\uu}^2 + \normacip{\uu}^2 + \normacurl{\uu}^2 \,, \\
\normam{\BB}^2 &:= 
\psm \Vert \BB \Vert^2 + \nm \Vert \bcurl (\BB) \Vert^2 \,.
\end{aligned}
\end{equation}

\begin{remark}\label{rem:thetaequiv}
In the definitions of the seminorm $\normacurl{\cdot}$ and the second term in $\normacip{\cdot}$, we have considered the approximation $\TT_h$ instead of $\TT$.
This choice was made to simplify the theoretical analysis of the method.
 We emphasize that controlling the norms with $\TT_h$, along with the $\boldsymbol L^2$-norm of $\vv_h$, ensures the control of the corresponding norms with $\TT$ as we show briefly here below (and also the converse holds true). In fact, using triangular inequality, an inverse estimate and standard approximation results together with inverse estimates for polynomials, we obtain 
\[
\begin{aligned}
\| \bcurl_h(\vv_h \times \TT_h) \|^2_{E} 
&\leq 
\| \bcurl_h(\vv_h \times \TT) \|^2_{E} + \| \bcurl_h(\vv_h \times (\TT_h - \TT)) \|^2_{E}  \\
&\lesssim
\| \bcurl_h(\vv_h \times \TT) \|^2_{E} + | \vv_h |_{1,E}^2\|\TT_h - \TT \|^2_{\boldsymbol{L}^\infty(E)} \\
& \qquad+
\| \vv_h \|_E^2 | \TT_h - \TT |^2_{\WW^1_\infty(E)}\\
& \lesssim\| \bcurl_h(\vv_h \times \TT) \|^2_{E} 
+ 
|\TT|^2_{\WW^1_\infty(E)}\|\vv_h\|^2_{E} \, ,  \\
\end{aligned}
\]
for all $\vv_h \in \ZZd$ (but, obviously, it holds also for all $\vv_h \in \VVd$). For the seminorm $\normacip{\cdot}$, we have that
\[
\begin{aligned}
\sum_{f\in\EdgesI} h_f^2 \| \jump{\bcurl(\vv_h \times \TT_h) } \|^2_{f}
&\leq
\sum_{f\in\EdgesI} h_f^2 \| \jump{\bcurl(\vv_h \times \TT) } \|^2_{f} \\
&\qquad+
\sum_{f\in\EdgesI} h_f^2 \| \jump{\bcurl(\vv_h \times (\TT_h-\TT)) } \|^2_{f}\, .
\end{aligned}
\]
Noting that
\[
\begin{aligned}
\sum_{f\in\EdgesI} h_f^2 \| \jump{\bcurl(\vv_h \times (\TT_h-\TT)) } \|^2_{f} 
\leq
\sum_{f\in\EdgesI} h_f^2 \bigl(&\| \bcurl(\vv_h \times (\TT_h-\TT))  \|^2_{f^+} 
\\
&\qquad + \| \bcurl(\vv_h \times (\TT_h-\TT))  \|^2_{f^-} \bigr) \, ,
\end{aligned}
\]
and using approximation estimates for $\TT \in \WW^1_\infty(\Omega_h)$, a scaled trace inequality and inverse estimates on $\vv_h$, we easily obtain
\begin{equation}\label{eq:tt-tth}
\begin{aligned}
\sum_{f\in\EdgesI} h_f^2 \| \jump{\bcurl(\vv_h \times (\TT_h-\TT)) } \|^2_{f}
&\lesssim
\sum_{f \in \EdgesI}h_f^2  
\| \nabla \vv_h \|^2_{\boldsymbol{L}^2(f^\pm)} \| \TT - \TT_h \|^2_{\boldsymbol{L}^{\infty}(f^\pm)} + \\
& \qquad+ \sum_{f \in \EdgesI}h_f^2\| \vv_h \|^2_{\boldsymbol{L}^2(f^\pm)} | \TT |^2_{\WW^1_{\infty}(f^\pm)}
\\
& \lesssim 
\sum_{E \in \Omega_h} h_E \| \vv_h \|^2_{E} |\TT|^2_{\WW^1_\infty(E)}
\lesssim h \,|\TT|^2_{\WW^1_\infty(\Omega_h)} \| \vv_h \|^2 \, .
\end{aligned}
\end{equation}
Hence
\[
\sum_{f\in\EdgesI} h_f^2 \| \jump{\bcurl(\vv_h \times \TT) } \|^2_{f}
\leq
\sum_{f\in\EdgesI} h_f^2 \| \jump{\bcurl(\vv_h \times \TT_h) } \|^2_{f}
+
 \sum_{E \in \Omega_h} h_E | \TT|^2_{\WW^1_\infty(\Omega_h)} \| \vv_h\|^2_E
\]
\end{remark}
The proof of the inf-sup stability is divided into two parts.
First, we control all the terms appearing in the previous definition, except for $\normacurl{\cdot}$, by testing the bilinear form $\astab(\cdot, \cdot, \cdot, \cdot)$ with a symmetric entry.
To handle the remaining term, we use a suitable interpolant of a function close to $\bcurl(\cdot \times \TT)$ and construct a function that satisfies a form of inf-sup condition with respect to such interpolant.

We omit the proof of the following result since it can be obtained with a standard argument in DG theory; we only underline that also Remark \ref{rem:thetaequiv} needs to be used to bound the $\normacip{\cdot}^2$ term in the right-hand side.

\begin{proposition}
\label{prp:coe}
Let the mesh assumptions \textbf{(MA1)}, and \textbf{(MA2)} if $k=1$. Assume the consistency assumption \textbf{(RA1)} holds.
If the parameter $\bdma$ in \eqref{eq:forme_d} is sufficiently large there exists a real positive constant $\ccoe$ such that for all $\vv_h \in \ZZd$ and $\HH_h \in \WWc$ the form $\astab(\cdot, \cdot, \cdot, \cdot)$ defined in \eqref{eq:astab} satisfies
\[
\astab(\vv_h, \BB_h, \vv_h, \BB_h) \geq 
\ccoe \left( \normas{\vv_h}^2 + \normaupw{\vv_h}^2 + \normacip{\vv_h}^2 + \normam{\BB_h}^2 \right) \, ,
\]
where the coefficient $\ccoe$ does not depend on the mesh size $h$ and on the problem parameters $\pss$, $\psm$, $\ns$, $\nm$, $\cc$, and $\TT$.
\end{proposition}
The following result aims to establish control over $\normacurl{\vv_h}$.
 We emphasize that, in the previous proposition, we managed to control all the terms in the norm except for 
$\normacurl{\vv_h}$.
\begin{proposition} \label{prp:preliminary inf sup}
Let the mesh assumptions \textbf{(MA1)}, and \textbf{(MA2)} if $k=1$. Assume that the consistency assumption \textbf{(RA1)} holds. 
Then for any $(\vv_h, \BB_h) \in \ZZd \times \WWd$ it exists $\HH_h \in \WWd$ such that
\begin{equation}\label{newprop:objective}
\begin{aligned}
\astab(\vv_h, \BB_h, 0, \HH_h) &\geq C_1\normacurl{v_h}^2 -C \Big(\bdmJJ^{-1} \normacip{\vv_h}^2 + (1 + \psm  h) \normam{\BB_h}^2
\\ & \qquad      +  \frac{| \TT |^2_{\WW^1_\infty(\Omega_h)^3} \,  h}{\pss} \| \vv_h \|^2_S 
+  \left(\frac{\bdmJ + \bdmJJ}{\bdmJ\,\bdmJJ} \right) \normacip{\vv_h}^2\Big) \, ,
\end{aligned}
\end{equation}
where the constants $C_1$ and $C$ does not depend on the mesh size $h$ and on the problem parameters $\pss$, $\psm$, $\ns$, $\nm$, $\cc$, and $\TT$.
\end{proposition}
\begin{proof}
We start by constructing $\HH_h$. We introduce $\pp_h\in\mathbb{O}_{k-1}(\Omega_h)$ as the interpolant of $\bcurl_h(\vv_h \times \TT_h)$ defined in Proposition~\ref{prp:oswald}.
Let $\hat\HH_h \in \Pkc_k(\Omega_h) \cap \WWc$ be the function constructed in \cite{BDV:2024} Lemma 4.2 that satisfies the following two inequalities:
\begin{equation}\label{eq:infsupHH}
\left\{
\begin{aligned}
&\sum_{E \in \Omega_h} h_E^{-2} \| \hat\HH_h \|_E^2 \lesssim \sum_{E \in \Omega_h} h_E^{2} \| \pp_h \|_E^2 \, , \\
&(\hat\HH_h, \pp_h) \gtrsim \sum_{E \in \Omega_h} h_E^{2} \| \pp_h \|_E^2 \, .
\end{aligned}
\right .
\end{equation}
Finally, we define $\HH_h = \gamma^{-1} \hat \HH_h$ and proceed to show that it satisfies \eqref{newprop:objective}. 

\noindent
We observe that, since the third entry is equal to zero, we easily obtain that
\begin{equation}\label{eq:equalToZero}
\pss (\vv_h, 0) = 
\ns \, \ash(\vv_h, 0) = c_h(\vv_h, 0) = d(\BB, 0)   = J_h(\vv_h, 0) = 0 \, .
\end{equation}
Now, we estimate the nonzero terms. Using Cauchy-Schwarz inequality, property \eqref{eq:infsupHH}, the stability of the interpolant in Proposition~\ref{prp:oswald}, and the regularity of the mesh, we obtain
\begin{equation}\label{eq:smh_estimate}
\begin{aligned}
| \psm (\BB_h, \HH_h) |
&\geq
- \psm \| \BB_h \| \| \HH_h \| \\
&\gtrsim 
-\psm\left( \sum_{E \in \Omega_h} h_E^2 \, \| \BB_h \|^2_E\right)^{\frac{1}{2}} \left( \sum_{E \in \Omega_h} h_E^{-2} \, \| \HH_h \|^2_E\right)^{\frac{1}{2}} \\
&=
-\psm\left( \sum_{E \in \Omega_h} h_E^2 \, \| \BB_h \|^2_E\right)^{\frac{1}{2}} \gamma^{-1}\left( \sum_{E \in \Omega_h} h_E^{-2} \, \| \Hat{\HH}_h \|^2_E\right)^{\frac{1}{2}} \\
&\gtrsim 
-\psm \, h^{\frac{1}{2}} \, \| \BB_h \| \gamma^{-\frac{1}{2}} \left( \sum_{E \in \Omega_h} h_E^{2} \, \| \pp_h \|^2_E\right)^{\frac{1}{2}} \\
&\gtrsim 
-\psm \, h^{\frac{1}{2}} \, \| \BB_h \| \normacurl{\vv_h} 
\gtrsim 
-\psm^{\frac{1}{2}} \, h^{\frac{1}{2}} \, \normam{\BB_h} \normacurl{\vv_h} \, .
\end{aligned}
\end{equation}
The bilinear form $\am(\BB_h,\HH_h)$ is estimated in a very similar way, by using also a polynomial inverse estimate and the fact that $\gamma \geq \nm$. We get

\begin{equation}\label{eq:amh_estimate}
\begin{aligned}
| \am (\BB_h, \HH_h) |
&\geq
- \nm \| \bcurl (\BB_h) \| \| \bcurl (\HH_h) \| \\
&\gtrsim 
-\nm\left( \sum_{E \in \Omega_h}  \| \bcurl (\BB_h) \|^2_E\right)^{\frac{1}{2}} \left( \sum_{E \in \Omega_h}  \| \bcurl (\HH_h) \|^2_E\right)^{\frac{1}{2}} \\
&\gtrsim
-\nm^{\frac{1}{2}}\normam{\BB_h} \left( \sum_{E \in \Omega_h} h_E^{-2} \, \| \HH_h \|^2_E\right)^{\frac{1}{2}} 
\!\gtrsim
-\nm^{\frac{1}{2}}\normam{\BB_h} \gamma^{-1}\!\left( \sum_{E \in \Omega_h} h_E^{-2} \, \| \hat{\HH}_h \|^2_E\right)^{\frac{1}{2}} \\
&\gtrsim 
-\nm^{\frac{1}{2}}\normam{\BB_h} \gamma^{-1}\left( \sum_{E \in \Omega_h} h_E^2 \, \| \bcurl_h(\vv_h \times \TT_h) \|^2_E\right)^{\frac{1}{2}} \gtrsim 
-\normam{\BB_h} \normacurl{\vv_h}\, .
\end{aligned}
\end{equation}
For the bilinear form $d(\HH_h,\vv_h)$, thanks to integration by parts, we have that 
\begin{equation}\label{eq:d-form_h}
\begin{aligned}
d(\HH_h, \vv_h) 
&= 
(\bcurl (\HH_h) \times \TT, \vv_h) \\
&=
(\bcurl_h (\TT \times \vv_h), \HH_h) 
+
\sum_{E \in \Omega_h}
(\HH_h \times \nn_E ,   \TT  \times \vv_h)_{\partial E} \, .
\end{aligned}
\end{equation}
The first term is split as
\[
\begin{aligned}
(\bcurl_h (\TT \times \vv_h), \HH_h) 
&=
(\bcurl_h ((\TT - \TT_h) \times \vv_h) , \HH_h) 
+
(\bcurl_h (\TT_h \times \vv_h) , \HH_h) \\
&=
(\bcurl_h ((\TT - \TT_h) \times \vv_h) , \HH_h) \\
& \qquad + (\bcurl_h (\TT_h \times \vv_h) - \pp_h, \HH_h) \\
& \qquad + (\pp_h, \HH_h) \\
& =: T_{\TT,1} + T_{\TT,2} + T_{\TT,3} \, .
\end{aligned}
\]
\textit{Estimate of $T_{\TT,1}$:} Thanks to the regularity of $\TT$, we have that
\begin{equation*}
\begin{aligned}
(\bcurl_h ((\TT - \TT_h) \times \vv_h) , \HH_h)
& \geq
- \left( \sum_{E \in \Omega_h} h_E^2 \|\bcurl_h ((\TT - \TT_h) \times \vv_h)\|_E^2 \right)^{\frac{1}{2}} \\
& \qquad \left( \sum_{E \in \Omega_h} h_E^{-2} \| \HH_h\|_E^2 \right)^{\frac{1}{2}} \, .
\end{aligned}
\end{equation*}
Using the vector calculus identity,
\begin{equation}\label{eq:curl-identity}
\bcurl(A \times B) = (\diver B) \, A - (\diver A) \, B + (\bnabla A) B - (\bnabla B) A \, ,
\end{equation}
and a polynomial inverse estimate combined with standard approximation results, we can easily obtain
\[
\|\bcurl_h ((\TT - \TT_h) \times \vv_h)\|^2_E \leq |\TT|_{\WW^1_\infty(\Omega_h)} \, \| \vv_h \|_E \, .
\]
Hence, using the same steps as the previous estimates, we have that
\[
\begin{aligned}
(\bcurl ((\TT - \TT_h) \times \vv_h) , \HH_h)
& \gtrsim
- \left( \sum_{E \in \Omega_h} h_E^2 |\TT|_{\WW^1_\infty(\Omega_h)} \, \| \vv_h \|_E \right)^{\frac{1}{2}} \\
& \qquad \gamma^{-1}\left( \sum_{E \in \Omega_h} h_E^2 \, \| \bcurl_h(\vv_h \times \TT_h) \|^2_E\right)^{\frac{1}{2}} \\
& \gtrsim
-
| \TT |_{\WW^1_\infty(\Omega_h)} h^{\frac{1}{2}} \| \vv_h \| \, \normacurl{\vv_h} \\
&\gtrsim
-
\left( \frac{| \TT |^2_{\WW^1_\infty(\Omega_h)} \,  h}{\pss}\right)^{\frac{1}{2}} \| \vv_h \|_S \, \normacurl{\vv_h} \, .
\end{aligned}
\]
\textit{Estimate of $T_{\TT,2}$:} Using Proposition \ref{prp:oswald}, the definition and properties of $\HH_h$, and recalling that $h_E \le h$, 
we obtain
\[
\begin{aligned}
(\bcurl_h (\TT_h \times \vv_h) - \pp_h, \HH_h)
& \gtrsim
-\left( \sum_{f \in \EdgesI} h_f^3 \| \jump{\bcurl_h(\TT_h \times \vv_h)}\|_f^2\right)^{\frac{1}{2}} \\
& \qquad \gamma^{-1} \left(  \sum_{E \in \Omega_h} h_E^{2} \, \| \bcurl_h(\vv_h \times \TT_h) \|^2_E\right)^{\frac{1}{2}} \\
&\gtrsim - \left(\frac{1}{\bdmJJ}\right)^{\frac{1}{2}} \normacip{\vv_h}\normacurl{\vv_h}
\, .
\end{aligned}
\]
\textit{Estimate of $T_{\TT,3}$:} Using the definition of $\HH_h$ and the properties of $\hat\HH_h$ yields
\[
(\pp_h, \HH_h) 
=
\gamma^{-1} (\pp_h, \hat\HH_h)
\gtrsim
\gamma^{-1} \sum_{E \in \Omega_h} h_E^2 \| \pp_h \|^2_E \, .
\]
Now, using triangular inequality, we observe that
\[
\|  \bcurl_h(\vv_h \times \TT_h) \|^2_E
\leq
2\| \pp_h \|^2_E 
+ 
2\|  \bcurl_h(\vv_h \times \TT_h)  - \pp_h\|^2_E \, ,
\]
which implies 
\[
\| \pp_h \|^2_E
\geq
\frac{1}{2}\|  \bcurl_h(\vv_h \times \TT_h) \|^2_E
-
\|  \bcurl_h(\vv_h \times \TT_h)  - \pp_h\|^2_E
 \, .
\]
Using Proposition \ref{prp:oswald}, we obtain
\[
\begin{aligned}
(\pp_h, \HH_h) 
&\geq
 \frac{1}{2}\gamma^{-1} \sum_{E \in \Omega_h} \left( h_E^2\|  \bcurl_h(\vv_h \times \TT_h) \|^2_E
- C h_E^3 \| \jump{\bcurl_h(\vv_h \times \TT_h)} \|^2_{\partial E} \right) \\
& \geq
\frac{1}{2}\normacurl{v_h}^2 - C \gamma^{-1} \bdmJJ^{-1} h \normacip{\vv_h}^2
 \geq \frac{1}{2}\normacurl{v_h}^2 - C \bdmJJ^{-1} \normacip{\vv_h}^2
\, ,
\end{aligned}
\]
where in the last inequality we simply used that $\gamma^{-1} h \le 1$ by the definition of $\gamma$.

Using Cauchy-Schwarz inequality, Proposition \ref{prp:oswald}, the definition of the jump operator, trace inequality, and \eqref{eq:infsupHH}, we have that
\[
\begin{aligned}
\sum_{E \in \Omega_h}
(\HH_h \times \nn_E ,   \TT  \times \vv_h)_{\partial E}
&\gtrsim -
\left(
\bdmJ^{-1} \sum_{f \in \Edges } \| \HH_h \|^2_f
\right)^{\frac{1}{2}}
\left( \bdmJ
\sum_{f \in \Edges } \| \jump{\vv_h \times \TT} \|^2_f
\right)^{\frac{1}{2}} \\
&\gtrsim -
\left(
\bdmJ^{-1}\sum_{E \in \Omega_h} h_E^{-1}\| \HH_h \|^2_E
\right)^{\frac{1}{2}}
\normacip{\vv_h} \\
&\gtrsim -
\gamma^{-\frac{1}{2}}\left(
\bdmJ^{-1}\sum_{E \in \Omega_h} h_E^{-2}\| \hat\HH_h \|^2_E
\right)^{\frac{1}{2}}
\normacip{\vv_h} \\
&\gtrsim -
\bdmJ^{-\frac{1}{2}}
\normacip{\vv_h}
\normacurl{\vv_h}\, .
\end{aligned}
\]
Gathering all the above estimates, we have obtained
\begin{equation}\label{newprop:almost}
\begin{aligned}
\astab(\vv_h, \BB_h, 0, \HH_h) &\geq \tilde C_1\normacurl{\vv_h}^2 - \tilde C_2\bdmJJ^{-1} \normacip{\vv_h}^2  \\
& \qquad - \tilde C_3\Big((1 + \psm  h) \normam{\BB_h}^2 
+
\frac{| \TT |^2_{\WW^1_\infty(\Omega)^3} \,  h}{\pss} \| \vv_h \|^2_S \\
& \qquad \qquad +  \Big(\frac{\bdmJ + \bdmJJ}{\bdmJ\,\bdmJJ} \Big) \normacip{\vv_h}^2\Big)^\frac{1}{2}\normacurl{v_h}^2 \, ,
\end{aligned}
\end{equation}
%
\\
where the constants $\tilde C_i$ do not depend on the mesh size $h$.
The thesis now follows by using Young inequality.
\end{proof}

In the following lemma, we prove that $\normam {\HH_h}$ is controlled by the norms of $\vv_h$.
\begin{lemma}[Continuity of the norm]
\label{lem:norm-cont}
Let the mesh assumptions \textbf{(MA1)}, and \textbf{(MA2)} if $k=1$. Assume that the consistency assumption \textbf{(RA1)} holds.
Then the test function $H_h$ introduced in Proposition \ref{prp:preliminary inf sup} satisfies 
\[
\normam{\HH_h}^2 
\lesssim
\dfrac{\psm}{\pss} \| \TT \|^2_{\boldsymbol{L}^\infty(\Omega_h)} \normas{\vv_h}^2 
+
\normacurl{\vv_h}^2 \,. 
\]
\end{lemma}

\begin{proof}
We need to estimate the two terms that appear in the definition of $\normam{\HH_h}$. For the first one, we have
\[
\begin{aligned}
\psm \| \HH_h\|^2
&=
\psm \, \gamma^{-2} \, \| \hat \HH_h\|^2
\leq
\psm \sum_{E} h_E^{-2} \| \hat \HH_h\|^2_E \\
&\lesssim
\psm \sum_{E} h_E^{2} \| \pp_h\|^2_E 
\lesssim
\psm \sum_{E} h_E^{2} \| \bcurl_h(\vv_h \times \TT_h)\|^2_E \\
&\lesssim
\psm \| \TT \|_{\boldsymbol{L}^\infty(\Omega_h)}^2 \| \vv_h\|^2 
\lesssim
\dfrac{\psm}{\pss} \| \TT \|^2_{\boldsymbol{L}^\infty(\Omega_h)} \normas{\vv_h}^2 \, .
\end{aligned}
\]
For the second term, we  use an inverse polynomial inequality, property \eqref{eq:infsupHH}, and the definition of $\gamma$
\[
\begin{aligned}
\nm \| \bcurl (\HH_h)\|^2
&\lesssim
\nm \sum_{E \in \Omega_h} h_E^{-2} \| \HH_h \|^2
=
\nm \, \gamma^{-2}\,\sum_{E \in \Omega_h} h_E^{-2} \| \hat\HH_h \|^2 \\
&\lesssim
\gamma^{-1}\,\sum_{E \in \Omega_h} h_E^{2} \| \bcurl_h(\vv_h \times \TT_h) \|^2_E = \normacurl{\vv_h}^2 \,. 
\end{aligned}
\]


\end{proof}
Finally, we conclude this section by establishing the inf-sup condition for our method.
\begin{theorem}\label{th:cip-inf-sup}
Let the mesh assumptions \textbf{(MA1)}, \textbf{(MA2)} if $k=1$. Assume that the consistency assumption \textbf{(RA1)} holds. 
We have that:
\begin{equation} \label{eq:infsup}
\normastab{\vv_h} + \normam{\BB_h} 
\lesssim 
\sup_{(\ww_h,\mathbf{K}_h) \in \ZZd \times \WWd} \dfrac{\astab(\vv_h, \BB_h, \ww_h, \mathbf{K}_h)}{\normastab{\ww_h} + \normam{\mathbf{K}_h}}
\qquad \text{$\forall (\vv_h,\BB_h) \in \ZZd \times \WWd$.}
\end{equation}
\end{theorem}
\begin{proof}
The proof becomes standard in light of the previous derivations.
It is sufficient to combine Propositions \ref{prp:coe} and \ref{prp:preliminary inf sup}
taking $(\ww_h,\mathbf{K}_h) = \kappa (\vv_h,\BB_h) +  (0,\HH_h)$ for a sufficiently large $\kappa$, in addition to recalling Lemma \ref{lem:norm-cont}.
Here, $\HH_h$ is defined as in Proposition \ref{prp:preliminary inf sup}.
\end{proof}

\subsection{Error analysis}
Let $(\uu, p, \BB)$ and $(\uu_h, p_h, \BB_h)$ 
be the solutions of \eqref{eq:linear variazionale} and \eqref{eq:linear fem} respectively. 
We introduce the following notations that will be useful for the error analysis
\begin{equation}
\label{eq:fquantities}
\eei := \uu - \uui\,, \qquad 
\eeh := \uu_h - \uui\,, \qquad 
\EEi := \BB - \BBi\,, \qquad 
\EEh := \BB_h - \BBi \,,
\end{equation}
where $\uui$ and $\BBi$ are defined in Lemma \ref{lm:int-v} and Lemma \ref{lm:int-w}. We define also the following quantity
\begin{equation}
\label{eq:quantities}
\begin{aligned}
\ls^2 &:= \max \biggl\{ 
\pss h^2 \,,  \max_{f \in \EdgesI} \Vert  \cc \cdot \nn_f \Vert_{\boldsymbol{L}^\infty(f)} h \,,
\Vert  \TT \Vert_{\WW^1_\infty(\Omega_h)}^2  h \,,
\ns
\biggr\} \,.
\end{aligned}
\end{equation}
We begin the error analysis with the following proposition.
\begin{proposition}\label{prp:abstrac}
Let $(\uu, \BB) \in \ZZc \times \WWc$ and $(\uu_h,  \BB_h) \in \ZZd \times \WWd$ be the solutions of \eqref{eq:linear variazionale} and \eqref{eq:linear fem} respectively. 
It holds that
\[
\normastab{\uu - \uu_h} + \normam{\BB - \BB_h}
\lesssim
\Xi_\mathcal{I}
+
\Xi_h \, ,
\]
where the two error quantities are defined as
\[
\Xi_{\mathcal I} := \normastab{\eei} + \normam{\EEi} \, ,
\]
and
\[
\Xi_h :=
\sup_{(\ww_h,\mathbf{K}_h) \in \ZZd \times \WWd} \dfrac{\astab(\uu, \BB, \ww_h, \mathbf{K}_h) -
\astab(\uui, \BBi, \ww_h, \mathbf{K}_h) }{\normastab{\ww_h} + \normam{\mathbf{K}_h}} \, .
\]
\end{proposition}
\begin{proof}
Using triangular inequality, we have that
\[
\normastab{\uu - \uu_h} + \normam{\BB - \BB_h}
\leq
\normastab{\eei} + \normam{\EEi}
+
\normastab{\eeh} + \normam{\EEh} \, .
\]
Applying Theorem \ref{th:cip-inf-sup} and the consistency property \eqref{eq:consistency}, we obtain
\[
\begin{aligned}
\normastab{\eeh} + \normam{\EEh}
&\lesssim
\sup_{(\ww_h,\mathbf{K}_h) \in \ZZd \times \WWd} \dfrac{\astab(\eeh, \EEh, \ww_h, \mathbf{K}_h)}{\normastab{\ww_h} + \normam{\mathbf{K}_h}} \\
&\lesssim
\sup_{(\ww_h,\mathbf{K}_h) \in \ZZd \times \WWd} \dfrac{\astab(\uu_h, \BB_h, \ww_h, \mathbf{K}_h) -
\astab(\uui, \BBi, \ww_h, \mathbf{K}_h) }{\normastab{\ww_h} + \normam{\mathbf{K}_h}} \\
&\lesssim
\sup_{(\ww_h,\mathbf{K}_h) \in \ZZd \times \WWd} \dfrac{\astab(\uu, \BB, \ww_h, \mathbf{K}_h) -
\astab(\uui, \BBi, \ww_h, \mathbf{K}_h) }{\normastab{\ww_h} + \normam{\mathbf{K}_h}} \, .
\end{aligned}
\]
The proof is concluded by recalling the definitions of $\Xi_\mathcal I$ and $\Xi_h$.
\end{proof}

To properly estimate the rate of convergence of our method, we make the following stronger assumptions. Note that we allow for magnetic fields not in $H^1(\Omega)$, since we are considering also non-convex lipschitz domains.

\smallskip
\noindent
\textbf{(RA2) Regularity assumptions on the exact solution (error analysis):} 

\noindent
Assume that:
\begin{itemize}
\item the velocity field satisfies
$\uu \in \HH^{\regu}(\Omega_h)$ 
for some $\frac{3}{2} < \regu \leq k+1,$
\item the magnetic field satisfies
$\BB \in \HH^{\regb}(\bcurl,\Omega_h)$ 
for some $\frac{1}{2} < \regb \leq k+1.$ 
\end{itemize}

\begin{lemma}[Interpolation error]
Let Assumption \textbf{(MA1)} hold.
Furthermore, if $k=1$ let also Assumption \textbf{(MA2)} hold. Then, under the regularity assumption \textbf{(RA2)}, it holds that
\[
\Xi_\mathcal I^2 \lesssim
\ls^2 \, h^{2\regu-2} \vert \uu \vert_{\regu, \Omega_h}^2
+
\psm h^{2\regb}   |\BB|^2_{\HH^\regb(\bcurl,\Omega_h)} 
+
\ns h^{2\tilde \regb}|\bcurl(\BB)|^2_{\tilde\regb,\Omega_h}\, ,
\]
where $\tilde \regb := \min \{ r,k \}$, see Lemma \ref{lm:int-w}. 
\end{lemma}
\begin{proof}
We recall that
\[
\Xi_\mathcal I = \normastab{\eei} + \normam{\EEi} \, .
\]
Exploiting the definition of $\normastab{\cdot}$ and following the same steps of \cite{BDV:2024}, we have that
\begin{align}
\label{eq:eiS}
\normas{\eei}^2 
& \lesssim
(\pss h^2 + \ns(1 + \bdma)) h^{2\regu-2} \vert \uu \vert_{\regu, \Omega_h}^2 \lesssim \ls^2 h^{2\regu-2} \vert \uu \vert_{\regu, \Omega_h}^2\,,
\\
\label{eq:eiupw}
\normaupw{\eei}^2 
& \lesssim
\bdmc \max_{f \in \EdgesI} \Vert \cc \cdot \nn_f \Vert_{\boldsymbol{L}^\infty(f)}
h^{2\regu-1} \vert \uu\vert_{\regu, \Omega_h}^2 \lesssim \ls^2 h^{2\regu-2} \vert \uu \vert_{\regu, \Omega_h}^2\, .
\end{align}
For the jump term in the norm, we have that
\[
\normacip{\eei}^2 =
\bdmJ\sum_{f \in \Edges}  \Vert \jump{\TT \times \eei}_f \Vert_{f}^2
+
\bdmJJ\sum_{f \in \EdgesI}  h_f^2 \Vert \jump{\bcurl_h(\eei\times \TT_h)}_f \Vert_{f}^2 \, .
\]
For the first term, using trace inequality and Lemma \ref{lm:int-w}, we obtain that
\begin{equation}\label{eq:eiJ1}
\begin{aligned}
\bdmJ \sum_{f \in \Edges} 
\|\jump{\TT \times \eei}_f \|^2_f  
&\lesssim
\bdmJ \sum_{f \in \Edges} \| \TT \|_{\boldsymbol{L}^\infty(f^\pm)}^2 \| \eei \|^2_f \\
&\lesssim
\bdmJ \sum_{f \in \Edges} \| \TT \|_{\boldsymbol{L}^\infty(f^\pm)}^2 \bigl(h_E^{-1}\| \eei \|^2_E + h_E| \eei |^2_{1,E} \bigr) \\
&\lesssim \bdmJ \max_{f \in \Edges}\| \TT \|_{\boldsymbol{L}^\infty(f^\pm)}^2 h^{2\regu-1} | \uu|^2_{\regu,\Omega_h} \lesssim \ls^2 h^{2\regu-2} | \uu|^2_{\regu,\Omega_h}\, .
\end{aligned}
\end{equation}
Similarly, on the other term, using \eqref{eq:curl-identity} together with the fact that $\TT_h$ is piecewise constant and $\diver (\eei)=0$, we obtain 
\begin{equation}\label{eq:eiJ2}
\begin{aligned}
\bdmJJ \sum_{f \in \EdgesI} 
h_f^2 \|\jump{\bcurl_h (\eei \times \TT_h)}_f\|^2
&\lesssim
\bdmJJ \sum_{f \in \EdgesI} h_f^2 \| \TT_h \|^2_{\boldsymbol{L}^{\infty}(f^\pm)} \| \nabla \eei \|^2_{\boldsymbol{L}^2(f^\pm)}
\lesssim
\ls^2 h^{2\regu-2}| \uu|^2_{\regu,\Omega_h} \, .   
\end{aligned}
\end{equation}
Combining \eqref{eq:eiJ1} and \eqref{eq:eiJ2}, we have
\[
\normacip{\eei}^2 \lesssim \ls^2 h^{2\regu-2} | \uu |_{\regu,\Omega_h}^2 \, .
\]
For $\normacurl{\eei}$, recalling \eqref{eq:curl-identity}, we have that
\[
\begin{aligned}
\normacurl{\eei}^2 &=
\gamma^{-1} \, \sum_{E \in \Omega_h} h_E^2 \, \| \bcurl_h(\eei \times \TT_h) \|^2_{E}
\lesssim
\gamma^{-1} \, \sum_{E \in \Omega_h} h_E^2 \| \TT \|_{L^{\infty}(\Omega_h)}^2 \| \nabla \eei \|^2_{E}\\
&\lesssim
h^{2\regu-1} \| \TT \|^2_{\boldsymbol{L}^\infty(\Omega_h)} | \uu |^2_{\regu,\Omega_h} \lesssim \ls^2 h^{2\regu-2} |\uu|_{\regu,\Omega_h}^2 \, .
\end{aligned}
\]
For the magnetic field, we have to use the interpolation estimate of Lemma \ref{lm:int-w}. It holds that
\[
\begin{aligned}
\normam{\EEi}^2  &\lesssim \sum_{E \in \Omega_h} \psm h^{2\regb}_E  |\BB|^2_{\HH^\regb(\bcurl,E)} 
+
\sum_{E \in \Omega_h}\nm h_E^{2\tilde \regb}|\bcurl(\BB)|^2_{\tilde\regb,E} \\
&\lesssim  \psm h^{2\regb}  |\BB|^2_{\HH^\regb(\bcurl,\Omega_h)} 
+
\nm h^{2\tilde \regb}|\bcurl(\BB)|^2_{\tilde\regb,\Omega_h}\, .
\end{aligned}
\]
\end{proof}

\begin{lemma}\label{lm:eh estimate} Let Assumption \textbf{(MA1)} hold.
Furthermore, if $k=1$ let also Assumption \textbf{(MA2)} hold. Then, under the regularity assumption  \textbf{(RA2)}. It holds that
\[
\begin{aligned}
\Xi_h^2
&\lesssim
(\ls^2 + \gs^2 + \gm^2) h^{2\regu-2} \vert \uu\vert^2_{\regu, \Omega_h}
 +
 \ns h^{2\tilde \regb}|\bcurl(\BB)|^2_{\tilde\regb,\Omega_h} \\
 & \quad +
(\psm h^{2\regb} + (\gs^2+\gm^2 +1 )
h^{2\regb-1} 
) \bigl |\BB|_{\HH^{\regb}(\bcurl,\Omega_h)}^2 
\, ,
\end{aligned}
\]
where
\begin{equation}
\label{eq:gamma}
\begin{aligned}
\gs^2 &:= \min\{ \pss^{-1} h^2,  \, \max\{\pss^{-1},\ns^{-1}\}h^4\} \Vert\cc\Vert_{\WW^1_\infty(\Omega_h)}^2 + \pss^{-1} h\Vert\TT\Vert_{\WW^1_\infty(\Omega_h)}^2\, ,\\
\gm^2 &:=
\min\{ \psm^{-1} h^2,  \, \nm^{-1} h^4 \} 
\Vert\TT\Vert_{\WW^1_\infty(\Omega_h)}^2 + \gamma h^{-1} \,.
\end{aligned}
\end{equation}
%
%
%
\end{lemma}

\begin{proof}
We recall that by definition
\[
\begin{aligned}
\astab(\uu, \BB, \ww_h, \mathbf{K}_h) &-
\astab(\uui, \BBi, \ww_h, \mathbf{K}_h)\\
&\qquad=
\bigl( \pss (\eei, \ww_h) + \ns \ash(\eei, \ww_h) \bigr)
+
\bigl(\psm (\EEi, \mathbf{K}_h) + \nm \am(\EEi, \mathbf{K}_h)  \bigr) 
\\  
& \qquad
+ c_h(\eei, \ww_h) -d(\EEi, \ww_h)   + d(\mathbf{K}_h, \eei) + J_h(\eei, \ww_h) 
=: \sum_{i=1}^6 \alpha_i \,.
\end{aligned}
\]
We need to estimate each of these six terms. Some of these terms are estimated similarly to \cite{BDV:2024}, so we only state the result.

\noindent
$\bullet$ Estimate of $\alpha_1$ and $\alpha_3$: using the same calculations of \cite{BDV:2024}, we have that
\begin{align}
\label{eq:alpha13-f}
\alpha_1 
\lesssim \ls \, h^{\regu-1} \vert \uu \vert_{\regu, \Omega_h} \normastab{\ww_h}
\, , \qquad
\alpha_3 \lesssim
(\ls + \gs) h^{\regu-1} \vert \uu\vert_{\regu, \Omega_h} 
\normastab{\ww_h} \,.
\end{align}
%
%

\noindent
$\bullet$ Estimate of $\alpha_2$: Using the interpolation estimate of Lemma \ref{lm:int-w}, we have that
\begin{equation}
\begin{aligned}
\bigl(\psm (\EEi, \mathbf{K}_h) + \nm \am(\EEi, \mathbf{K}_h)  \bigr)  
&\leq 
\psm \| \EEi \| \|\mathbf{K}_h\| 
+
\nm \| \bcurl(\EEi) \| \, \|\bcurl(\mathbf{K}_h)\| \\
&\lesssim
\bigl(\psm^\frac{1}{2} h^{\regb} \bigl |\BB|_{\HH^{\regb}(\bcurl,\Omega_h)} 
+
\nm^\frac{1}{2} h^{\tilde \regb}  |\bcurl(\BB)|_{\tilde \regb,\Omega_h} \bigr)\normam{\mathbf{K}_h}
\, .
\end{aligned}
\end{equation}
\noindent
$\bullet$ Estimate of $\alpha_4$: 
Following \cite{BDV:2024}, it holds that by integration by parts
\begin{equation}
\label{eq:alpha-4}
\begin{aligned}
\alpha_4 & = 
(\EEi , \bcurl_h (\ww_h  \times (\TT - \TT_h))) +
(\EEi , \bcurl_h (\ww_h  \times \TT_h)) +
\\
& +
\sum_{f \in \Edges}
(\EEi \times \nn_f , \jump{\ww_h  \times \TT})_f 
=: \sum_{j=1}^3 \alpha_{4,j} \, .
\end{aligned}
\end{equation}
For the first term, mimicking the steps of \cite{BDV:2024} we have that
\begin{equation}
\label{eq:alpha4-1}
\begin{aligned}
\alpha_{4,1}
\lesssim
\gs
h^{\regb-\frac{1}{2}} \vert \BB \vert_{\HH^{\regb}(\bcurl,\Omega_h)} 
\normas{\ww_h} \, .
\end{aligned}
\end{equation}
Thanks to the definition of $\normacurl{\cdot}$, we obtain
\[
\begin{aligned}
\alpha_{4,2} = ( \EEi \,, \bcurl_h (\ww_h  \times \TT_h) )
&\lesssim
\left( \sum_{E \in \Omega_h} h_E^{-2} \| \EEi\|^2_E\right)^\frac12
\left( \sum_{E \in \Omega_h} h_E^{2} \| \bcurl_h (\ww_h  \times \TT_h)\|^2_E\right)^\frac12 \\
&\lesssim
\gamma^\frac12 \left( \sum_{E \in \Omega_h} h_E^{-2} \| \EEi\|^2_E\right)^\frac12
\normacurl{\ww_h} \\
& \lesssim \gamma^{\frac{1}{2}} h^{\regb-1} | \BB|_{\HH^r(\bcurl,\Omega_h)}\normacurl{\ww_h} \\
&\lesssim {\gm} h^{\regb-\frac12} | \BB |_{\HH^{\regb}(\bcurl,\Omega_h)} \normacurl{\ww_h} \, .
\end{aligned}
\]
Finally,
\begin{equation}
\label{eq:alpha4-3}
\begin{aligned}
\alpha_{4,3} & \leq
\biggl(\bdmJ^{-1} \sum_{f \in \Edges}
\Vert \EEi \Vert_{f}^2 \biggr)^{1/2} \normacip{\ww_h}
 \lesssim
\biggl(\bdmJ^{-1} \sum_{E \in \Omega_h}
(h_E^{-1} \Vert \EEi \Vert_{E}^2 + 
h_E^{2\hat\regb-1} \vert \EEi \vert_{\hat\regb,E}^2) \biggr)^{1/2} \normacip{\ww_h}
\\
& \lesssim
h^{\regb-\frac{1}{2}} |\BB|_{\HH^\regb(\bcurl,\Omega_h)} \normacip{\ww_h}
\lesssim
h^{\regb-\frac{1}{2}} \vert \BB \vert_{\HH^\regb(\bcurl,\Omega_h)}
\normacip{\ww_h}
\end{aligned}
\end{equation}
where $\hat\regb := \min\{\regb,1\}$.

\noindent
$\bullet$ Estimate of $\alpha_5$: Exploiting the orthogonality property of the interpolation operator, and standard approximation results, we obtain
\begin{equation}
\label{eq:alpha5-f}
\begin{aligned}
\alpha_5 & = 
(\bcurl (\mathbf K_h) \times \TT \,, \eei )
= (\bcurl (\mathbf K_h) \times (\TT-\TT_h) \,, \eei  )
\\
& \leq \sum_{E \in \Omega_h}
\Vert \TT - \TT_h \Vert_{\boldsymbol{L}^\infty(E)}
\Vert \bcurl (\mathbf K_h) \Vert_{E} 
\Vert \eei \Vert_{E}
\\
& \lesssim \sum_{E \in \Omega_h}
h_E \Vert \TT \Vert_{\WW^1_\infty(E)}
\Vert \bcurl (\mathbf K_h) \Vert_{E} 
\Vert \eei \Vert_{E}
\\
& \lesssim 
\Vert \TT \Vert_{\WW^1_\infty(\Omega_h)}
\Vert \eei \Vert
\biggl(\sum_{E \in \Omega_h}
h_E^2 
\Vert \bcurl (\mathbf K_h) \Vert_{E}^2 \biggr)^{1/2} 
\\
& \lesssim 
\min \{ \psm^{-1/2},  \nm^{-1/2} h\} 
\Vert \TT \Vert_{\WW^1_\infty(\Omega_h)}
h^{\regu} \vert \uu \vert_{\regu, \Omega_h}
\normam{\mathbf K_h} 
\\
& \lesssim 
\gm h^{\regu-1} \vert \uu \vert_{\regu, \Omega_h}
\normam{\mathbf K_h} \, .
\end{aligned}
\end{equation}

\noindent
$\bullet$ Estimate of $\alpha_6$: 
By definition, we have that
\[
\begin{aligned}
\alpha_6 &= \sum_{f \in \Edges} \bigl( \bdmJ 
( \jump{\TT \times \eei}_f , \jump{\TT \times \ww_h}_f \bigr)_f \\
& \qquad +
\bdmJJ h_f^2 
\bigl(\jump{\bcurl_h (\eei \times \TT)}_f  ,\jump{\bcurl_h (\ww_h \times \TT)}_f 
\bigr)_f =: \sum_{j=1}^2 \alpha_{6,j} \, .
\end{aligned}
\]
On the first one, using \eqref{eq:eiJ1}, we have that
\[
\begin{aligned}
\alpha_{6,1}
&=
\bdmJ\sum_{f \in \Edges} \bigl( \jump{\TT \times \eei}_f , \jump{\TT \times \ww_h}_f \bigr)_f \lesssim
\normacip{\eei} \normacip{\ww_h} \lesssim
\ls h^{\regu-1} |\uu|_{\regu,\Omega_h}\normastab{\ww_h} \, .
\end{aligned}
\]
Using the Cauchy-Schwarz and triangular inequalities, estimates in the same spirit of \eqref{eq:tt-tth}, and bound \eqref{eq:eiJ2}, we obtain (without showing all the details)
\[
\begin{aligned}
\alpha_{6,2}
&\leq
\bdmJJ \sum_{f \in \EdgesI} 
h_f^2 \bigl(\jump{\bcurl_h (\eei \times \TT)}_f  ,\jump{\bcurl_h (\ww_h \times \TT)}_f 
\bigr)_f \\
& \lesssim
\left(\bdmJJ \sum_{f \in \EdgesI} 
h_f^2 \bigl(\|\jump{\bcurl_h (\eei \times (\TT-\TT_h))} \|^2_f + \|\jump{\bcurl_h (\eei \times \TT_h)} \|^2_f\bigr) \right)^{\frac{1}{2}}\normastab{\ww_h} \\
& \lesssim
\left(\normacip{\eei}^2 
+ \sum_{E \in \Omega_h} \| \TT \|_{\WW^1_\infty(E)}^2  
\big( h_E \|\eei\|^2_E + h_E^{2\regu+1} |  \eei |^2_{\regu,E} \big)
\right)^{\frac{1}{2}}\normastab{\ww_h}\\
&\lesssim
\ls h^{\regu-1} |\uu|_{\regu,\Omega_h}\normastab{\ww_h} \, .
\end{aligned}
\]
Note that, differently from \eqref{eq:tt-tth} and since $\eei$ is not a discrete function, here above we applied a scaled trace inequality from $\HH^1(f)$ to $\HH^s(E)$, with $f$ face of element $E$ (we recall $s>3/2$ by assumption), instead of an inverse inequality. \\
%
\end{proof}
Collecting the previous result we obtain the final estimate.
\begin{theorem}\label{thm:convergence}
Let Assumption \textbf{(MA1)} hold. Furthermore, if $k=1$ let also Assumption \textbf{(MA2)} be valid. Then, under the regularity assumption  \textbf{(RA2)} it holds
\[
\begin{aligned}
\normastab{\uu - \uu_h}^2 + \normam{\BB - \BB_h}^2
&\lesssim
(\ls^2 + \gs^2 + \gm^2) h^{2\regu-2} \vert \uu\vert^2_{\regu, \Omega_h}
 +
 \ns h^{2\tilde \regb}|\bcurl(\BB)|^2_{\tilde\regb,\Omega_h} \\
 & \quad +
(\psm h^{2\regb} + (\gs^2+\gm^2+1 )
h^{2\regb-1}
) \bigl |\BB|_{\HH^{\regb}(\bcurl,\Omega_h)}^2 
\, .
\end{aligned}
\]
\end{theorem}

We now study the error on the pressure variable. We first recall that, exploiting the inf-sup condition for the $\boldsymbol{{\rm BDM}}$ elements and the Poincarè inequality, there exists $\ww_h \in \VVd$ such that
\begin{equation}
\label{eq:p-infsup}
\Vert \ww_h \Vert_{1,h} \lesssim 1 \,, 
\qquad \text{and} \qquad
\Vert p_h - \Pi_{k-1} p\Vert 
\lesssim b(\ww_h \,, p_h - \Pi_{k-1} p) \,.
\end{equation}
where the discrete norm is defined by:
$$
\Vert \ww_h \Vert_{1,h}^2 := \sum_{E \in \Omega_h} \Vert \ww_h \Vert_{1,E} 
+ \sum_{f \in \Edges}  h_f^{-1} \Vert \jump{\ww_h}_f \Vert_{f}^2 \, .
$$
Furthermore, thank to the inclusion $\diver(\VVd) \subseteq \mathbb P_{k-1}(\Omega_h)$, it also holds
\begin{equation}
\label{eq:p-orth}
b(\vv_h \,, p) =
b(\vv_h \,, \Pi_{k-1} p)
\qquad 
\text{for all $\vv_h \in \VVd$.}
\end{equation}
We are now ready to prove the following error estimates.
\begin{theorem}[Error estimates for the pressure]
\label{thm:conv pre}
Under the same assumptions as in Theorem~\ref{thm:convergence}, we have that
\begin{equation}\label{p-est-first}
\begin{aligned}
\Vert p - p_h \Vert
&\lesssim 
(\ts + \ls + \Vert \cc \Vert_{\boldsymbol{L}^\infty(\Omega)} \pss^{-1/2} )
\normastab{\uu - \uu_h}   +
\nm^{-\frac{1}{2}} \normam{\BB - \BB_h}
+ 
h^k \, \vert p \vert_{k, \Omega_h} \, , 
\end{aligned}
\end{equation}
where $\ts^2 := \max\{\ns, \pss\}.$
Furthermore, assuming that
\begin{equation}\label{eq:jump-assumption}
\Vert \jump{\TT} \Vert_{\boldsymbol{L}^\infty(f)}\lesssim h_f \Vert \TT \Vert_{\WW^1_\infty(\omega_f)} \quad \forall f \in \Edges \, ,
\end{equation}
with $\omega_f$ denoting the union of the elements sharing the face $f$, the following improved estimate (avoiding the term $\nm^{-\frac{1}{2}}$) holds true:
\begin{equation}\label{p-est-second}    
\begin{aligned}
\Vert p - p_h \Vert
&\lesssim 
(\ts + \ls + \Vert \cc \Vert_{\boldsymbol{L}^\infty(\Omega)} \pss^{-1/2} )
\normastab{\uu - \uu_h} + h^k \, \vert p \vert_{k, \Omega_h} 
 \\
& 
 + \Vert \TT \Vert_{\WW^1_\infty(\Omega_h)} \big( \psm^{-1/2}\normam{\BB - \BB_h} + h^{r} \Vert \BB \Vert_{\HH^{r}(\bcurl,\Omega)} \big)  \, .
\end{aligned}
\end{equation}
\end{theorem}
\begin{proof}
We only sketch the proof since it follows standard arguments in this type of problems and analogous steps with respect to \cite{BDV:2021}. We first recall that from Lemma \ref{lm:bramble} it holds
\begin{equation}\label{eq:proj-err}
\| p - \Pi_{k-1} p \|
\lesssim
h^{k} | p |_{k,\Omega_h} \, .
\end{equation}
Exploiting estimates \eqref{eq:p-infsup} and equation \eqref{eq:p-orth} in Problems \eqref{eq:linear variazionale} and \eqref{eq:linear fem}, we obtain
\begin{equation}
\label{eq:p-error}
\begin{aligned}
\Vert p_h -  \Pi_{k-1} p\Vert  &
\lesssim
b(\ww_h \,, p_h - \Pi_{k-1} p)
=
b(\ww_h \,, p_h - p)
\\
& =
\bigl( \pss(\uu - \uu_h, \ww_h) + \ns \ash(\uu - \uu_h, \ww_h) \bigr)+
c_h(\uu - \uu_h, \ww_h) + \\  
&
-d(\BB - \BB_h, \ww_h)    + J_h(\uu - \uu_h, \ww_h)
 =: \sum_{i=1}^4 \beta_i \,.
\end{aligned}
\end{equation}
Following Lemma \ref{lm:eh estimate} and \cite{BDV:2024}, we infer
\begin{equation}
\label{eq:beta}
\begin{aligned}
\beta_1 & \lesssim \ts 
\normastab{\uu - \uu_h} \,,
\\
\beta_2 & \lesssim (\ls 
+ 
\Vert \cc \Vert_{\boldsymbol{L}^\infty(\Omega)} \pss^{-1/2}
)\normastab{\uu - \uu_h} \,,
\\
\beta_4 & \lesssim 
\ls
\normastab{\uu - \uu_h} \, .
\end{aligned}
\end{equation}
The term $\beta_3$ can be directly estimated as
\begin{equation}\label{eq:beta3_plain}
\beta_3 
= 
( \bcurl ((\BB-\BB_h)) \times \TT ,\, \ww_h)_\Omega
\lesssim
\nm^{-\frac{1}{2}} \normam{\BB - \BB_h} \, .
\end{equation}
Therefore, estimate \eqref{p-est-first} follows from \eqref{eq:proj-err}, \eqref{eq:p-error}, \eqref{eq:beta} and \eqref{eq:beta3_plain}.
%
%
To prove estimate \eqref{p-est-second}, we treat the term $\beta_3$ in a different way, as we show briefly here below. We start by noting that
\begin{equation}\label{jump-ave}
\jump{\ww_h \times \TT}_f
=
\jump{\ww_h}_f \times \media{\TT}_f
+
\media{\ww_h}_f \times \jump{\TT}_f \, .
\end{equation}
After integrating by parts and applying \eqref{jump-ave} we obtain 
$$
\begin{aligned}
\beta_3 & = \sum_{E \in \Omega_h} (\BB-\BB_h, \, \bcurl(\ww_h \times \TT))_E  
+ \sum_{f \in \Edges} ((\BB-\BB_h)\times \nn_f, \, \jump{\ww_h}_f \times \media{\TT}_f)_f \\
& + \sum_{f \in \Edges} ((\BB-\BB_h)\times \nn_f, \, \media{\ww_h}_f \times \jump{\TT}_f)_f 
=: \sum_{i=1}^3 \beta_{3,i} \, .
\end{aligned}
$$
We estimate the three terms above using H\"older and trace inequalities. The first two addenda are handled easily, using $\TT \in \WW^1_\infty(\Omega_h)$ and $\Vert \ww_h \Vert_{1,h} \lesssim 1$ (thus we avoid showing the details). The third term is bounded by 
$$
\beta_{3,3} \leq \sum_{f \in \Edges} h_f^{1/2} \Vert \BB-\BB_h \Vert_{\boldsymbol{L}^2(f^\pm)} h_f^{1/2}
\Vert \ww_h \Vert_{\boldsymbol{L}^2(f^\pm)} h_f^{-1} \Vert \jump{\TT} \Vert_{\boldsymbol{L}^\infty(f)} \, ,
$$
which by trace inequalities, using \eqref{eq:jump-assumption} and again $\Vert \ww_h \Vert_{1,h} \lesssim 1$ yields
$$
\beta_{3,3} \lesssim \Vert \TT \Vert_{\WW^1_\infty(\Omega_h)} \left( 
\sum_{E \in \Omega_h} \Vert \BB-\BB_h \Vert_E + h_E^2 | \BB-\BB_h |_{1,E}^2
\right)^{1/2} \, .
$$
We bound the term here above by first adding and subtracting $\BBi$, then using a triangle inequality, finally applying approximation estimates for $\BB - \BBi$ and inverse inequalities for $\BBi - \BB_h$. Such calculations lead to 
$$
\begin{aligned}
\beta_{3,3} 
&\lesssim 
\Vert \TT \Vert_{\WW^1_\infty(\Omega_h)} 
\big(  \|\BBi - \BB_h\| + h^r \Vert \BB \Vert_{\HH^r(\bcurl,\Omega)} \big) \\
& \lesssim
\Vert \TT \Vert_{\WW^1_\infty(\Omega_h)} 
\big(  \|\BB - \BB_h\|
+
\|\BB - \BBi\| + h^r \Vert \BB \Vert_{\HH^r(\bcurl,\Omega)} \big) \\
& \lesssim
\Vert \TT \Vert_{\WW^1_\infty(\Omega_h)} 
\big(  \psm^{-\frac{1}{2}}\normam{\BB - \BB_h}+ h^r \Vert \BB \Vert_{\HH^r(\bcurl,\Omega)} \big)  \, .
\end{aligned}
$$
%
Estimate \eqref{p-est-second} now follows from \eqref{eq:proj-err}, \eqref{eq:p-error}, \eqref{eq:beta} and the above bounds for $\beta_3$.

\end{proof}

\begin{remark}
Assumption \eqref{eq:jump-assumption} means that $\TT$ does not jump too much across faces, which makes sense if we consider $\TT$ as a discrete approximation of a regular function. Obviously, this hypothesis is satisfied if $\TT$ is a continuous function.    
\end{remark}

For completeness, we close this section with the following convergence result without requirements on the analytical solution regularity.
\begin{proposition} 
Let Assumption (MA1) hold. Let $\{{\bf u}_h, p_h, {\bf B}_h\}_h$ denote a sequence of solutions of the discrete problem \eqref{eq:linear fem} with mesh parameter $h$ tending to zero, and let $({\bf u}, p, {\bf B}) \in \VVc \times \Qc \times \WWc$ be the solution of the continuous problem \eqref{eq:linear variazionale}. Then, for vanishing $h$, it holds (for any $2 \leq p < 6$)
$$
\uu_h \to \uu \quad \textrm{in } L^p(\Omega), \qquad p_h \rightharpoonup p \quad \textrm{weakly in } L^2(\Omega) \, ,
$$
$$
\BB_h \to \BB \quad \textrm{in } L^2(\Omega), \qquad \BB_h \rightharpoonup \BB \quad \textrm{weakly in } \HH(\bcurl, \Omega) \, .
$$
\end{proposition}
\begin{proof}
The proof of this result very closely mimics the one of Proposition 5.4 in \cite{da2023robust}. Thus, we only underline the main difference: in order to handle the convergence of the magnetic field $\BB_h$, instead of using the compact inclusion of $\HH^1(\Omega)$ into $\boldsymbol L^p(\Omega)$, $2 \leq p < 6$, we use the discrete compactness result in \cite{boffi2001note,boffi2000fortin}, which holds also thanks to Remark \ref{rem:Bh:kernel}. 
\end{proof}

\section{Numerical results}
\label{sec:numerical}
In this section, we develop a set of numerical tests for our method, with the aim of comparing the actual converge rates with the rates expected by the theory. Furthermore, we compare our approach with the one proposed in \cite{BDV:2024} in a non-convex domain.
The proposed space--time method has been implemented taking inspiration from the \texttt{C++} library \texttt{Vem++}~\cite{Dassi:2023:VAC}.

In the first test we consider a cubic domain $\Omega = [0,1]^3$, and in the second test an extruded L-shaped domain $\Omega = [-1,1]^3 \setminus ([-1,0)^2 \times [-1,1])$. In all cases we use a family of unstructured tetrahedral meshes of characteristic size $h$.
Figure \ref{fig:meshes} illustrates two sample meshes, one for each domain.

\begin{figure}[!htb]
\begin{center}
\begin{tabular}{ccc}
\texttt{(a) Convex domain} & &\texttt{(b) Non-convex domain} \\
\includegraphics[width=0.35\textwidth]{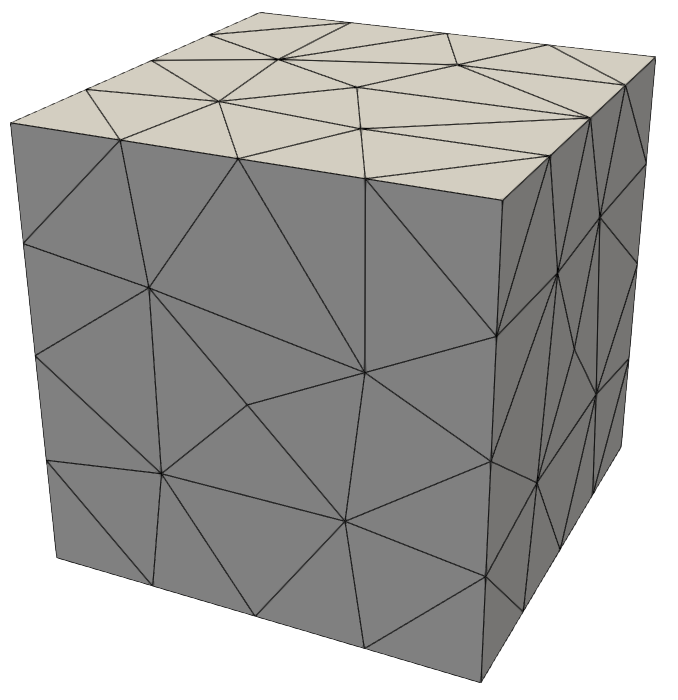} &\phantom{mm}&
\includegraphics[width=0.37\textwidth]{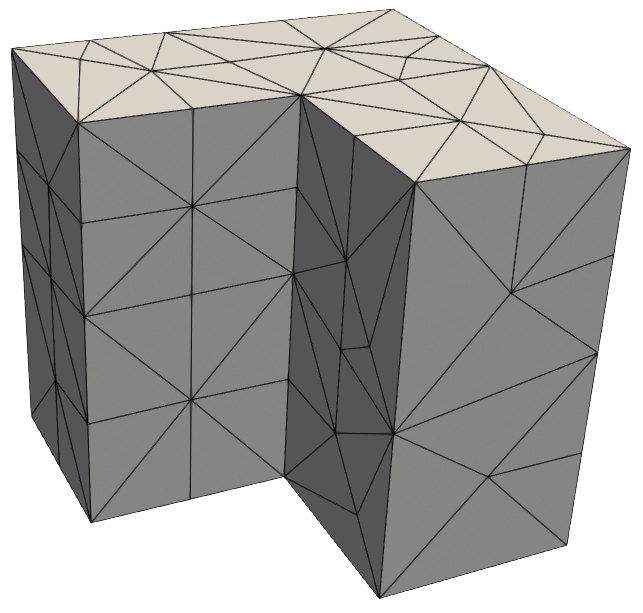} \\
\end{tabular}
\end{center}
\caption{Example of meshes used for the tests (Unit cube and L-shaped domain).}
\label{fig:meshes}
\end{figure}

We consider the following error quantities:
\begin{itemize}
    \item $H^1(\Omega)$-seminorm of the velocity field;
    \item $H(\bcurl,\Omega)$-seminorm of the magnetic field;
    \item $L^2(\Omega)$ norm of the pressure;
    \item a "total" norm given by
    \[
    \begin{aligned}
    \|(\uu-\uu_h,\BB-\BB_h) \|_{\text{total}} &:= \pss \| \uu - \uu_h \| + \ns \|\bnabla(\uu - \uu_h) \| \\
    & +\psm \| \BB - \BB_h \| + \nm \|\bcurl(\BB - \BB_h) \| \\
    &
    +
    \bdmJ\sum_{f \in \Edges}  \Vert \jump{\TT \times (\uu-\uu_h)}_f \Vert_{f}^2 \\
    &+
    \bdmJJ\sum_{f \in \EdgesI}  h_f^2 \Vert \jump{\bcurl_h((\uu-\uu_h)\times \TT)}_f \Vert_{f}^2 \, .
    \end{aligned}
    \]
\end{itemize}

\paragraph{Test 1: convergence study for a regular case.} 
In this paragraph, we analyze the convergence of the error with respect to the parameter $h$. 
We consider the convex domain $\Omega = [0,1]^3$ and the polynomial degrees $k=1,2$.
We design a problem for which the fields
\[
\uu(x,y,z)
:=
\begin{bmatrix}
\sin(\pi x) \cos(\pi y) \cos(\pi z)\\
\cos(\pi x) \sin(\pi y) \cos(\pi z)\\
-2 \cos(\pi x) \cos(\pi y) \sin(\pi z)
\end{bmatrix} ,
\qquad
\BB(x,y,z) 
:=
\begin{bmatrix}
\sin(\pi y) \\
\sin(\pi z) \\
\sin(\pi x)
\end{bmatrix},
\]
\[
p(x,y,z) = \sin(\pi x) + \sin(\pi y) - 2 \sin(\pi z)
\]
are solutions of \eqref{eq:linear variazionale},
where the functions $\cc$ and $\TT$ are set to
\[
\cc = \uu, \qquad \TT = \BB .
\]
Two choices for the parameters $\nu = \nm = \ns$ are considered, while the reaction terms are always set to $\sigma = \pss = \psm = 1$.
In particular, we consider a diffusion-dominated case corresponding to $\nu = 1$, and a convection-dominated case corresponding to $\nu = 10^{-6}$.
Finally, we set $\bdmJ = 0.05$, $\bdmJJ = 0.01$ and $\mu_a = 10$ or $\mu_a = 20$ depending on $k=1$ or $k=2$.
In Figure \ref{fig:ex1_nu1} and Figure \ref{fig:ex1_nu1e-6}, the numerical results are shown. 
We can observe that all the quantities converge at least with the expected order, see Theorems \ref{thm:convergence} and \ref{thm:conv pre}, and that the scheme is clearly robust with respect to the parameter $\nu$. In particular, the total error exhibits (at least) the expected $h^{1/2}$ gain in the pre-asymptotic reduction rate in the convection dominated case.

\begin{figure}[!htb] 
\begin{center}
\includegraphics[width=\textwidth]{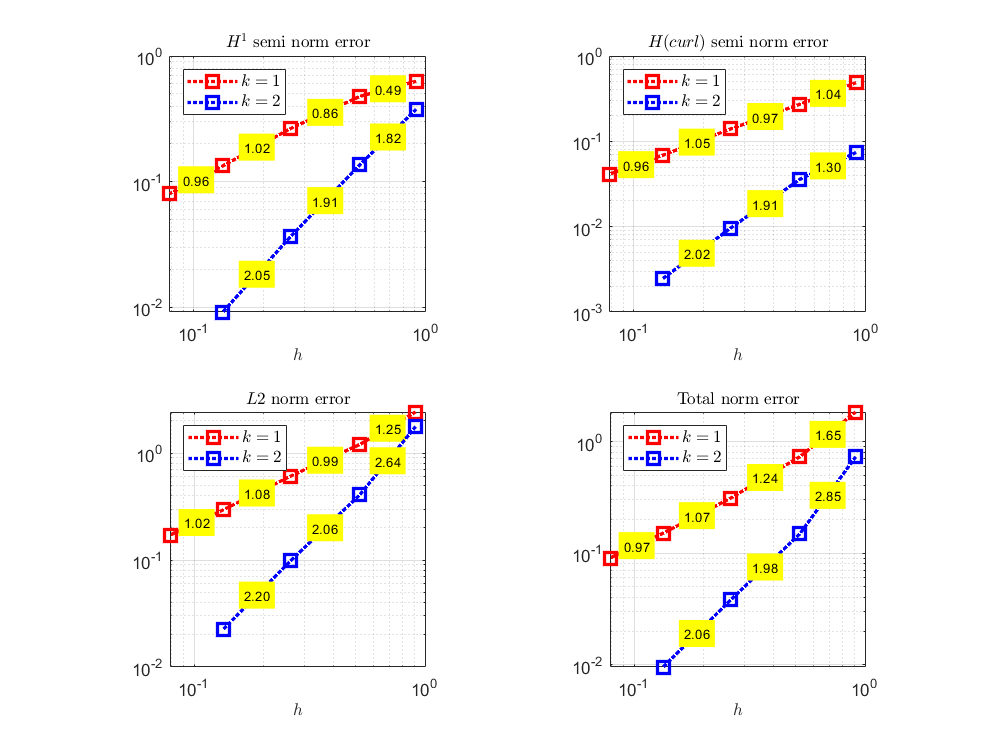}
\end{center}
\caption{Numerical results for Test 1, diffusion dominated case ($\nu = 1$).}
\label{fig:ex1_nu1}
\end{figure}

\begin{figure}[!htb]
\begin{center}
\includegraphics[width=\textwidth]{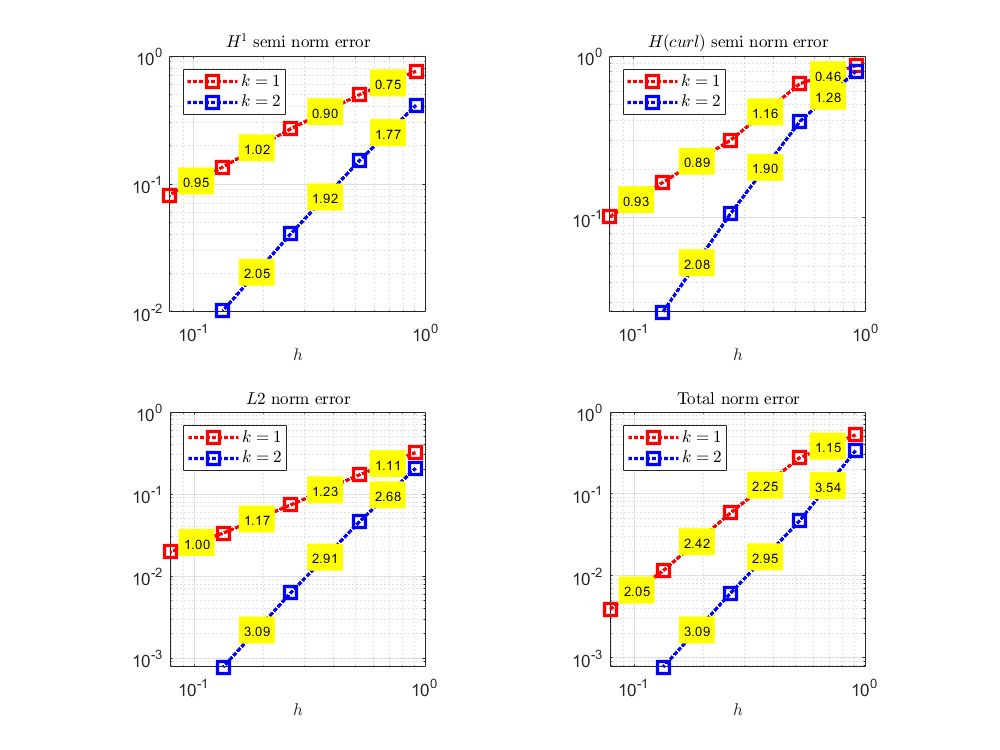}
\end{center}
\caption{Numerical results for Test 1, convection dominated case ($\nu = 10^{-6}$).}
\label{fig:ex1_nu1e-6}
\end{figure}

\paragraph{Test 2: a benchmark on a non-convex domain.} 
In this paragraph, we consider the non-convex domain depicted in Figure \ref{fig:meshes} (extruded L-shaped domain $\Omega = [-1,1]^3 \setminus ([-1,0)^2 \times [-1,1])$). The method proposed in \cite{BDV:2024} was devised for convex domains and, as such, makes use of $\HH^1$-conforming discrete magnetic fields. Therefore, in particular, it requires the exact magnetic field ${\bf B}$ to be in $\HH^1(\Omega)$. We here develop a numerical benchmark in order to 
\begin{enumerate}
    \item check from the practical perspective this limitation of the scheme in \cite{BDV:2024};
    \item show that the modified scheme here presented is suitable also for this kind of less regular problems. 
\end{enumerate}
Clearly, this comes at the price of a slightly higher dimensional system (when comparing the two schemes for the same order $k$).

We consider the solution of \eqref{eq:linear variazionale} given by the fields
\[
\uu(x,y,z) :=
\begin{bmatrix}
y^2 \\
z^2 \\
x^2
\end{bmatrix} \, ,
\qquad
\cc :=
\begin{bmatrix}
1 \\
2 \\
-1
\end{bmatrix} \, ,
\qquad
\TT :=
\begin{bmatrix}
1 \\
-1 \\
2
\end{bmatrix} \, .
\]
To construct the non-smooth solution for the magnetic field, we define the function
\[
r(x,y,z) = \sqrt{x^2+y^2}^\frac{2}{3} \sin\left(\frac{2}{3}\bigl(\arctan(y/x) + \frac{\pi}{2} \bigr)\right) \, ,
\]
and than we take 
\[
\BB(x,y,z) = \bnabla r(x,y,z) \, .
\]
We note that since $\BB$ is the gradient of a function, it holds $\bcurl(\BB) = 0.$
Furthermore $\BB \in [\HH^{\frac{2}{3}}(\Omega)]^3$ but $\BB \not\in [\HH^{1}(\Omega)]^3$.
We consider the same parameters of the previous example and we set $k=1$.
For the pressure, we simply consider
\[
p(x,y,z) = 0 \, .
\]
The results are shown in Figure \ref{fig:ex2_nu1} and Figure \ref{fig:ex2_nu1e-6}.
We note that the approach proposed in \cite{BDV:2024} fails to reach the optimal rate of convergence in the non-convex domain, the results becoming even less reliable for small values of the parameter $\nu$.
In fact, we observe that in many graphs the errors increase when reducing $h$.
Contrary, our approach converge as expected.

\begin{figure}[!htb]
\begin{center}
\includegraphics[width=\textwidth]{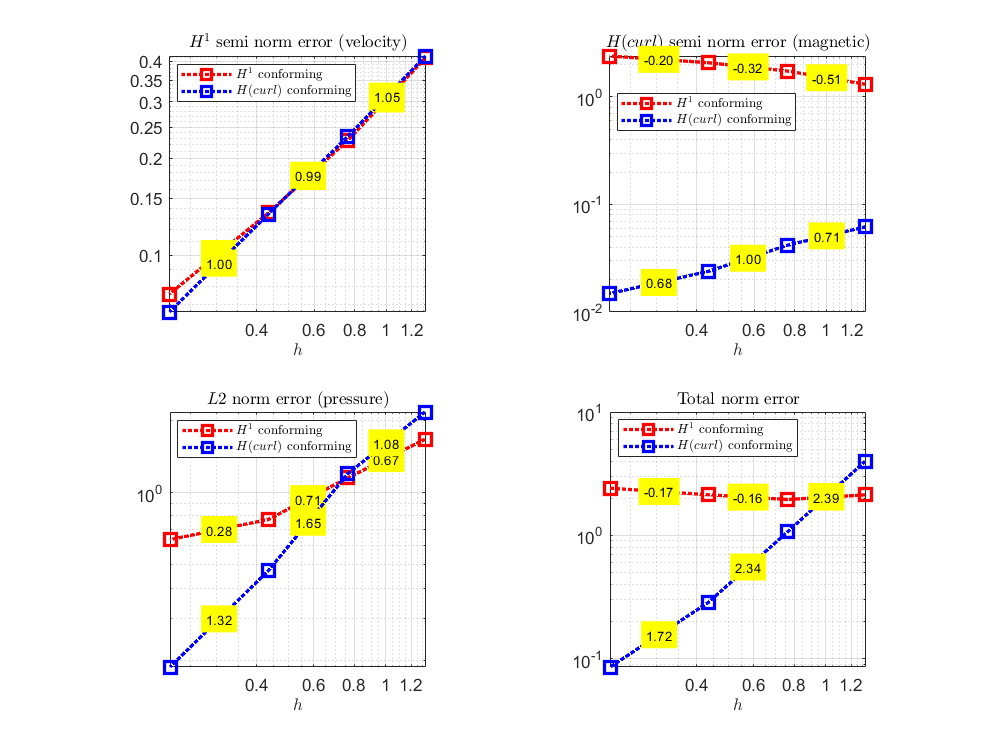}
\end{center}
\caption{Numerical results for the second test case corresponding to $\nu = 1$.}
\label{fig:ex2_nu1}
\end{figure}

\begin{figure}[!htb]
\begin{center}
\includegraphics[width=\textwidth]{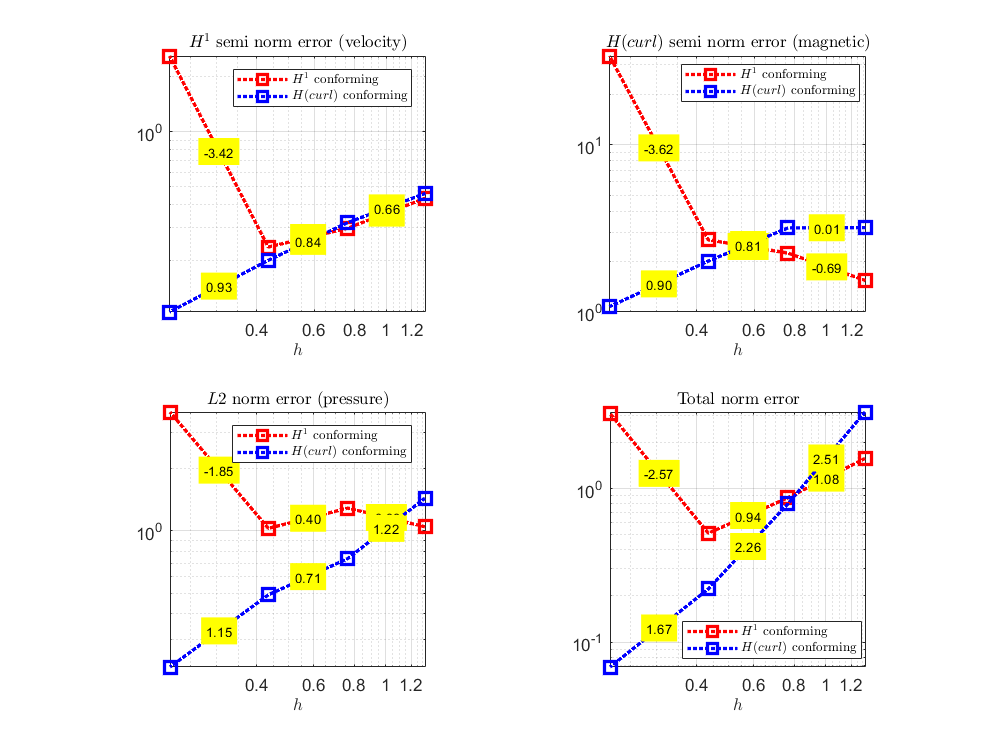}
\end{center}
\caption{Numerical results for the second test case corresponding to $\nu = 10^{-6}$.}
\label{fig:ex2_nu1e-6}
\end{figure}

\smallskip
\begin{center}
{\bf Aknowledgements} 
\end{center}

\medskip
LBDV and MT have been partially funded by the European Union (ERC Synergy, NEMESIS, project number 101115663).
Views and opinions expressed are however those of the authors only and do not necessarily reflect those of the European Union or the ERC Executive Agency. 
The authors have been partially supported by the INdAM Research group GNCS.
The authors are also indebted to Franco Dassi, from the University of Milano Bicocca, for his help and support in the initial stages of the method implementation.

\addcontentsline{toc}{section}{\refname}
\bibliographystyle{plain}
\bibliography{biblio,references}

\end{document}